\documentclass[a4paper,twoside,11pt,leqno]{amsart}
\usepackage{fullpage}
\usepackage{amsmath} 
\usepackage{amssymb}
 \usepackage{amsopn}
\usepackage{amsthm} 
\usepackage{amsfonts} 
\usepackage{mathrsfs}
\usepackage{mathabx}
\usepackage{stmaryrd} 
\usepackage{cite}
\usepackage[foot]{amsaddr}

\usepackage{verbatim}

\usepackage[T2A,T1]{fontenc}

\usepackage[pdfpagelabels,pdftex]{hyperref}
\usepackage[dvips]{graphicx} \usepackage[usenames]{color}
\usepackage[all]{xy}

\usepackage{setspace}
\usepackage{amssymb}
\usepackage{euscript}
\usepackage{cite}
\usepackage[all]{xy}
\usepackage{tabularx}

\theoremstyle{plain}
\usepackage{amsfonts}
\usepackage{graphicx}
\usepackage{amsmath}

\hypersetup{
  pdftitle={},
  pdfauthor={},
  pdfsubject={},
  pdfkeywords={},
  colorlinks=false,
  breaklinks=true,
  bookmarksopen=true,
  bookmarksnumbered=true,
  pdfpagemode=UseOutlines,
  plainpages=false}

\newcommand{\bA}{{\mathbb A}}

\newcommand{\bF}{{\mathbb F}}
\newcommand{\bG}{{\mathbb G}}

\newcommand{\bQ}{{\mathbb Q}}

\newcommand{\bZ}{{\mathbb Z}}

\newcommand{\cB}{{\mathscr B}}
\newcommand{\cC}{{\mathscr C}}

\newcommand{\caO}{{\mathcal O}}

\newcommand{\fA}{{\mathfrak A}}

\newcommand{\fI}{{\mathfrak I}}
\newcommand{\fJ}{{\mathfrak J}}

\newcommand{\fM}{{\mathfrak M}}

\newcommand{\fd}{{\mathfrak d}}

\newcommand{\fp}{{\mathfrak p}}

\DeclareMathOperator{\GL}{GL}

\newcommand{\tr}{{\rm tr}}

\newcommand{\matzz}[4]{\left(
\begin{array}{cc} #1 & #2 \\ #3 & #4 \end{array} \right)}

\DeclareMathOperator{\Gal}{Gal}

\newcommand\reallywidehat[1]{%
\savestack{\tmpbox}{\stretchto{%
  \scaleto{%
    \scalerel*[\widthof{\ensuremath{#1}}]{\kern-.6pt\bigwedge\kern-.6pt}%
    {\rule[-\textheight/2]{1ex}{\textheight}}
  }{\textheight}%
}{0.5ex}}%
\stackon[1pt]{#1}{\tmpbox}%
}

\makeatletter
\newtheorem*{rep@theorem}{\rep@title}
\newcommand{\newreptheorem}[2]{%
\newenvironment{rep#1}[1]{%
 \def\rep@title{#2 \ref{##1}}%
 \begin{rep@theorem}}%
 {\end{rep@theorem}}}
\makeatother

\newtheorem{thm}{Theorem}[section]

\newtheorem{prop}[thm]{Proposition}

\newtheorem{cor}[thm]{Corollary}

\newtheorem{lm}[thm]{Lemma}

\newreptheorem{thm}{Theorem}
\newreptheorem{prop}{Proposition}
\newreptheorem{cor}{Corollary}

\theoremstyle{definition}
\newtheorem{Def}[thm]{Definition}

\newtheorem{quest}[thm]{Question}

\newtheorem{rem}[thm]{Remark}

\newenvironment{pro*}[1][Proof]{{\it{#1:}} }{}

\newcommand\cInd{\mathop{ \rm c-Ind}}

\newcommand\rar{ \rightarrow }
\newcommand\tar{ \twoheadrightarrow }
\newcommand\har{ \hookrightarrow }

\newcommand\LRar{ \Leftrightarrow }
\newcommand\longrar{\longrightarrow}

\newcommand\Res{\mathop{ \rm Res}}
\newcommand\ord{\mathop{\rm ord}}

\newcommand\charac{\mathop{ \rm char}}

\newcommand{\sm}{{\,\smallsetminus\,}}
\newcommand\N{{\rm N}}
\newcommand\Tr{{\rm Tr}}

\newcommand\coh{{\rm H}}

\newcommand\Ind{\mathop{ \rm Ind}}

\newcommand{\uniff}{\varpi}
\newcommand{\unife}{\pi}

\newcounter{absatzcounter}[section]
\numberwithin{equation}{section}

\begin{document}

\title{Ordinary $GL_2(F)$-representations in characteristic two via affine Deligne-Lusztig constructions}
\author{Alexander B. Ivanov}
\email{ivanov@math.uni-bonn.de}
\date{February 7, 2018}


\begin{abstract}
The group $\GL_2$ over a local field with (residue) characteristic $2$ possesses much more smooth supercuspidal $\ell$-adic representations, than over a local field of residue characteristic $> 2$. One way to construct these representations is via the theory of types of Bushnell-Kutzko. We construct many of them in the cohomology of certain extended affine Deligne-Lusztig varieties attached to $\GL_2$ and wildly ramified maximal tori in it. Then we compare our construction with the type-theoretic one. The corresponding extended affine Deligne-Lusztig varieties were introduced in a preceding article. Also in the present case they turn out to be zero-dimensional.
\end{abstract}

\maketitle

\section{Introduction}
This note is a follow-up of the two previous papers \cite{Ivanov_15_unram, Ivanov_15_ram} studying coverings of extended affine Deligne-Lusztig varieties for $\GL_2$ over a local field (of equal characteristic). Here we analyze representations of $G = \GL_2$ over a local field $F$ of characteristic $2$ attached to a wildly ramified torus. Fix a wildly ramified Galois extension $E/F$ of degree $2$ and relative discriminant $\fd_{E/F} = \fp_F^{d+1}$, and an embedding $\Res_{E/F} \mathbb{G}_m \har G$. On $F$-points this induces an embedding $\iota \colon E^{\times} \har G(F)$. Each maximal minisotorpic wildly ramified torus of $G$ comes via such an embedding and each two embeddings attached to the same extension $E$ are conjugate by $G(F)$, so the various $E$ parametrize the $G(F)$-conjugacy classes of maximal minisotorpic wildly ramified tori in $G$.

We study the attached extended affine Deligne-Lusztig varieties as defined in \cite{Ivanov_15_ram}. They turn out to be zero-dimensional. Our main result is that all (admissible, smooth and irreducible) minimal supercuspidal representations $\pi$ of $\GL_2(F)$ of arbitrary deep normalized level $\ell(\pi)$, attached to $\iota(E^{\times}) \subseteq G(F)$ through the theory of types of Bushnell-Kutzko (see \cite{BushnellH_06}), satisfying $2\ell(\pi) \geq 3d$, occur in the cohomology of these coverings. The condition $2\ell(\pi) > 3d$ forces those representations to be \emph{ordinary}, i.e., they lie in the image of the \emph{imprimitive} Langlands-correspondence (\cite{BushnellH_06} \S44.1). This simply means that the corresponding representation of the Weil group is induced from the Weil group of a degree two (wildly ramified) extension. Conversely, a big portion of the (minimal) ordinary representations does satisfy the condition $\ell(\pi) > 3d$ (\cite{BushnellH_06} \S45.2 Theorem).

More precisely, these representations in a fixed level $\frac{m+d}{2}$ (for an appropriate $m \geq 2d$) will appear as a family parametrized by certain characters of an abelian group, called $\widetilde{\Gamma}/\Gamma^{\prime}$ below (an extension of $E^{\times}/U_E^{m+1}$; see Section \ref{sec:group_acting_on_right}), that is we will construct a map (see \eqref{eq:def_of_Rtheta})
\begin{align*} 
\{ \text{generic characters of } \widetilde{\Gamma}/\Gamma^{\prime} \} &\rightarrow \left\{ \begin{aligned} \text{Isomorphism classes of smooth} \\ \text{irreducible admissible $G(F)$-representations} \end{aligned} \right\} \\
\tilde{\theta} &\mapsto R_{\tilde{\theta}} := \coh_c^0(X_{\dot{w}}^m(1),\overline{\mathbb{Q}}_{\ell})[\tilde{\theta}].
\end{align*}
where $[\tilde{\theta}]$ denotes the $\tilde{\theta}$-isotypic component with respect to a natural action of $\widetilde{\Gamma}/\Gamma^{\prime}$ on a certain extended affine Deligne-Lusztig variety $X_{\dot{w}}^m(1)$.

An interesting fact to point out is that in (a slight reformulation of) \cite{BushnellH_06}, the same representations are also parametrized by characters of an abelian group, which easily can be extracted from the theory of types (see Section \ref{sec:comparison_BKH}, in particular Lemma \ref{lm:char_of_Pi_from_theta_psi_alpha}). We denote it here by $\Pi$. The relation between the two parametrizations is given by a somewhat exotic isomorphism $\beta \colon \widetilde{\Gamma}/\Gamma^{\prime} \stackrel{\sim}{\rar} \Pi$ (see Proposition \ref{prop:beta} and Remark \ref{rem:about_the_exotic_iso}). Here is a simplified version of our main result.

\begin{thm}[see Corollary \ref{cor:irred_level_minimality} and Theorem \ref{thm:relation_ADLV_BH}] Let $\tilde{\theta}$ be a generic character of $\widetilde{\Gamma}/\Gamma^{\prime}$. The $G(F)$-representation $R_{\tilde{\theta}}$ is admissible, irreducible, minimal and supercuspidal. Moreover, if $\beta^{\vee}(\tilde{\theta}) \in \Pi^{\vee}$ is the corresponding character of $\Pi$, and ${\rm BH}_{\beta^{\vee}(\tilde{\theta})}$ denotes the corresponding representation attached via the theory of types, one has
\[
R_{\tilde{\theta}} \cong {\rm BH}_{\beta^{\vee}(\tilde{\theta})}.
\]
\end{thm}

To show the first part of the theorem, the second part, i.e., the comparison with Bushnell-Kutzko types is not necessary. In the heart of the proofs of both parts are certain trace computations on the geometric and the type-theoretic side, see Section \ref{sec:hard_traces}.

As the main results here and in \cite{Ivanov_15_ram} indicate, the way how the extended affine Deligne-Lusztig varieties from \cite{Ivanov_15_ram} realize the ``automorphic induction'' of characters of $F$-points of maximal minisotorpic \emph{ramified} tori to $G(F)$ is quite near to the theory of Bushnell-Kutzko types. A nice consequence of this is the fact that it gives a geometric realization of the theory of types in (highly) ramified cases. A less clear consequence is that it seems to be further away from the Galois side than one might hope, see Section \ref{sec:rectifier}. However, it is still an open and interesting question, whether there is a twist of the actions on the geometric objects in the style of \cite{Weinstein_09} Section 5, which establishes a connection to the Galois side.

In Section \ref{sec:remark_on_tame_case} we discuss a slight simplification for the proof of the main result in \cite{Ivanov_15_ram}, which is concerned with the similar construction for $\GL_2$ and a purely tamely ramified torus. 

\subsection*{Acknowledgments} I would like to thank Guy Henniart, Laurent Fargues, Jean-Fran\c{c}ois Dat and Daniel Kirch for the fruitful discussions during my stay in Paris. I'm especially grateful to Guy Henniart, who made several helpful suggestions concerning this work. This work was written during the author's stay at the University Paris 6 (Jussieu). It was supported by a postdoctoral research fellowship of the German Research Foundation (DFG). 


\section{Automorphic induction from wild tori in $\GL_2$} \label{sec:extended_ADLV_and_Gamma_tilde_etc}

In this section we assume that $F$ has characteristic $2$ and put $G = \GL_2$. 

\subsection{Notations and preliminaries}\label{sec:not_and_prel}

We need to fix more notation. 
For a local non-archimedean field field $L$, denote by $\caO_L$ its integers, by $\fp_L$ its maximal ideal, $U_L = \caO_L^{\times}$, by $U_L^m$ the $m$-units of $L$, and by $\ord_L$ its valuation, normalized such that that it takes value $1$ on an uniformizer. For integers $a < b$, we introduce a shortcut notation:
\[
(\fp_L^a/\fp_L^b)^{\ast} := (\fp_L^a / \fp_L^b) \sm (\fp_L^{a+1} / \fp_L^b).
\]

\subsubsection{Arithmetical data} \label{sec:basic_notation}

We let $F$ be a local non-archimedean field of characteristic $2$ with residue field $k = \bF_q$. We let $E$ be a totally (wildly) ramified extension of $F$ of degree $2$ and discriminant $\fp_F^{d+1}$ for some $d > 0$ odd. By Artin-Schreier theory, we may choose uniformizers $\uniff$ resp. $\unife$ of $F$ resp. $E$, such that $\unife$ satisfies the minimal equation over $F$,  
\[
\unife^2 + \Delta \unife + \uniff = 0,
\]
for some $\Delta \in F$ with $\ord_F(\Delta) = \frac{d+1}{2}$. Concretely, we have $F = k((\uniff))$, $\caO_F = k\llbracket \uniff \rrbracket$, $E = k((\unife))$, $\caO_E = k\llbracket \unife \rrbracket$. 

We denote by $\N_{E/F} \colon E^{\times} \rar F^{\times}$ resp. by $\tr_{E/F} \colon E \rar F$ the norm resp. the trace map of $E/F$. We denote the maps induced by $\tr_{E/F}$ on subquotients of $\caO_E$ and $\caO_F$ again by $\tr_{E/F}$.  We let $\tau$ be the unique non-trivial element of the Galois group of $E/F$. We have $\tr_{E/F}(\unife) = \tau(\unife) + \unife = \Delta$ and $\N_{E/F}(\unife) = \unife \tau(\unife) = \uniff$. 

Set $\varepsilon := \frac{\tau(\unife)}{\unife}$. Then $\varepsilon \in U_E^d \sm U_E^{d+1}$. For convenience we will write $\varepsilon = 1 + \unife^d \varepsilon_0$ with $\varepsilon_0 \in U_E$. Clearly, $\varepsilon_0 = \unife^{-(d+1)}\Delta$.

\subsubsection{Group-theoretical data}\label{sec:group_th_data}
We set $G = \GL_2$ throughout this section. We denote by $Z$ the center of $G$, by $T$ the (split) diagonal torus and by $W$ the Weyl group of $G,T$. Further, $\widetilde{W} = X_{\ast}(T) \rtimes W$ denotes the extended affine Weyl group and $W_{\rm aff} \subseteq \widetilde{W}$ the affine Weyl group. The latter is a Coxeter group and we denote by $\ell(\cdot)$ the length function on it. 

Consider the embedding of $F$-algebras 
\[
\iota \colon E \har {\rm Mat}_{2 \times 2}(F), \quad \unife \mapsto \matzz{\Delta}{1}{\uniff}{0}
\]
\noindent We denote again by $\iota$ its restriction to the embedding $\iota \colon E^{\times} \har G(F)$. Its image are the $F$-points of a maximal minisotropic torus of $G$, which is split by $E$.
  

\subsubsection{Properties of trace and norm}

We will make use of the following well-known facts:
\begin{lm}\label{lm:trace_norm} 
\begin{itemize}
\item[(i)][\cite{BushnellH_06} 41.2]  Let $k \in \bZ$. We have $\tr_{E/F}(\fp_E^k) = \fp_F^{\ell}$, where $\ell = \lfloor \frac{k+d+1}{2}\rfloor$
\item[(ii)] Let $x \in E^{\times}$. Then $x^{-1}\tau(x) \in U_E^d$. If $x \in U_E^k$ with $k \geq 0$, then $x^{-1}\tau(x) \in U_E^{k + d}$.
\end{itemize}
\end{lm}
\begin{proof}
(i): see \cite{BushnellH_06} 41.2. (ii): If $k > 0$ and $x\in U_E^k$, write $x = 1 + \unife^k y$ for some $y \in \caO_E$. One computes
\begin{eqnarray*} 
x^{-1}\tau(x) = \frac{1 + \unife^k\varepsilon^k\tau(y)}{1 + \unife^k y} = (1 + \unife^k\varepsilon^k \tau(y))\sum_{i=0}^{\infty} (\unife^k y)^i = 1 + (y + \varepsilon^k \tau(y)) \sum_{i=1}^{\infty} \unife^{ki} y^{i-1}. 
\end{eqnarray*}
As $\varepsilon \equiv 1 \mod \fp_E^d$, part (i) of the lemma shows that  $(y + \varepsilon^k \tau(y)) \equiv 0 \mod \fp_E^d$. If $k = 0$, we may write $x = x_0(1+\unife^n y)$ with $x_0 \in k$ and some $n > 0$ and $y\in \caO_E$, and then apply the already proven case $k > 0$. This shows the second statement of (ii). If $x \in E^{\times}$, we may write $x = \unife^v x_0$ with $x_0 \in U_E$ and $v \in \bZ$ and thus $x^{-1}\tau(x) = \varepsilon^v x_0^{-1}\tau(x_0)$. Now the first statement of (ii) follows from the second.
\end{proof}


\subsubsection{Bruhat-Tits buildings} \label{sec:BTB}
For $L = F$ or $L = E$, we denote by $\mathscr{B}_L$ the Bruhat-Tits building of $G$ over $L$ and by $\mathscr{A}_L \subseteq \mathscr{B}_L$ the apartment of $T$. There is a natural $\tau$-action on $\mathscr{B}_E$. Moreover, there is a natural embedding $\iota \colon \mathscr{B}_F \har \mathscr{B}_E$ in the sense of \cite{Rousseau_77} Definition 2.5.1. We identify $\cB_F$ and $\iota(\cB_F)$ as sets but regard $\iota(\cB_F)$ provided with the structure of a sub-simplicial complex of $\cB_E$. As $E/F$ is ramified, each alcove of $\cB_F$ contains exactly two alcoves of $\iota(\cB_F)$. In particular, we observe that there is a bijection 
\begin{equation}\label{eq:bijection_alcoves_vertices}
(\text{alcoves of $\cB_F$}) \stackrel{1:1}{\leftrightarrow} (\text{vertices of $\iota(\cB_F)$, which are not vertices of $\cB_F$})
\end{equation}

\noindent As the ramification of $E/F$ is wild, we have $\iota(\mathscr{B}_F) \subsetneq \mathscr{B}_E^{\langle \tau \rangle}$, due to the occurrence of the so-called \emph{barbs}. 

We have $\mathscr{A}_E = \mathscr{A}_F$ as sets, and each alcove of $\mathscr{A}_F$ contains exactly two alcoves of $\mathscr{A}_E$. Fix a base alcove $\underline{a}_0$ in $\mathscr{A}_E$ and let $I_E \subseteq G(E)$ be its stabilizer. It is an Iwahori subgroup of $G(E)$. Further, we have the corresponding Iwahori subgroup $I_F := I_E \cap G(F)$ of $G(F)$. The alcove $\underline{a}_0$ has precisely one vertex, which is also a vertex of an alcove of $\mathscr{A}_F$. This is a hyperspecial vertex in $\mathscr{B}_E$ and we specify the splitting
 $\widetilde{W} = X_{\ast}(T) \rtimes W$ as being attached to this point. We denote by $P_{1/2}$ the vertex of $\underline{a}_0$, which is not a vertex of $\cB_F$. 


\subsubsection{Root subgroups}
Let $\Phi = \{+, -\}$ be the set of roots of $T$ in $G$, $+$ (resp. $-$) being the root contained in the upper (resp. lower) triangular Borel subgroup. We may regard $0$ as a root, with root subgroup $T (\cong \bG_m^2)$. For $\ast \in \Phi \cup \{0\}$, denote by
\[
e_{\ast} \colon U_{\ast} \rar G
\]
the embedding of the root subgroup. Thus for $a \in E$, we have $e_+(a) = \matzz{1}{a}{0}{1}$, $e_0(c,d) = \matzz{c}{0}{0}{d}$, etc.


\subsubsection{Level subgroups} \label{sec:level_subgroups} For $n \geq 0$ and $L = F$ or $L = E$ define the normal subgroups
\[
I_L^n :=  \matzz{1 + \fp_L^{\lfloor \frac{n+1}{2} \rfloor}}{\fp_L^{\lfloor \frac{n}{2} \rfloor}}{\fp_L^{\lfloor \frac{n}{2} \rfloor + 1}}{1 + \fp_L^{\lfloor \frac{n+1}{2} \rfloor}} \subseteq \matzz{\caO_L^{\times}}{\caO_L}{\fp_L}{\caO_L^{\times}} = I_L 
\]
(we choose $\underline{a}_0$ in Section \ref{sec:BTB} such that $I_E$ has this form). This coincides with the notation in \cite{BushnellH_06}. In particular, for $m \geq 0$, 
\[ 
I_E^{2m+1} = \matzz{1 + \fp_E^{m+1}}{\fp_E^m}{\fp_E^{m+1}}{1 + \fp_E^{m+1}} 
\]


\subsubsection{Vertex of departure} Let $\mathscr{C} \subseteq \mathscr{B}_E$ be a connected non-empty subcomplex. Let $C$ be an alcove of $\mathscr{B}_E$, which is not contained in $\mathscr{C}$. There is a unique gallery $\Gamma = (C_0, C_1, \dots, C_d)$ of minimal length $d$, such that $C_0 = C$ and $C_d$ is not contained in $\cC$, but has a (unique) vertex contained in $\cC$. This vertex is called the \emph{vertex of departure} of $C$ from $\cC$ (as in \cite{Reuman_02}). We define the \emph{distance} $C$ from $\cC$ to be the integer ${\rm dist}(C; \cC) := d+1$. For $C$ in $\cC$, we define the distance ${\rm dist}(C; \cC)$ to be zero.


\subsubsection{Barbs in $\cB_E$} 

\begin{prop}\label{prop:barbs}
The subcomplex $\cB_E^{\langle \tau \rangle}$ is equal to the closure of the union of all alcoves of $\cB_E$ with distance $\leq d$ to $\iota(\cB_F)$.
\end{prop}

\begin{proof}
We have $\cB_F \subseteq \cB_E$. Let $D$ be an alcove in $\cB_E$, not contained in $\cB_F$. Let ${\rm depart}(D; \iota(\cB_F))$ be the vertex of departure of $D$ from $\iota(\cB_F)$. We claim that the vertex ${\rm depart}(D; \iota(\cB_F))$ of $\iota(\cB_F)$ is not a vertex of $\cB_F$. Indeed, let a vertex $P$ of $\cB_F$ be given. The number of all alcoves of $\iota(\cB_F)$ having $P$ as a vertex is exactly $q + 1$ and the same is true for the number of alcoves of $\cB_F$, having $P$ as a vertex. Thus any alcove of $\iota(\cB_F)$ which has $P$ as a vertex, is necessarily contained in an alcove of $\cB_F$, which shows our claim. 

We have to show that an alcove $D$ in $\cB_E$, which does not lie in $\iota(\cB_F)$, lies in $\cB_E^{\langle \tau \rangle}$ if and only if its distance to $\iota(\cB_F)$ is $\leq d$. Using \eqref{eq:bijection_alcoves_vertices}, the transitive action of $G(F)$ on alcoves of $\cB_F$ and the above claim, we may assume that ${\rm depart}(D; \iota(\cB_F))$ is $P_{1/2}$. Let $v \in W_{\rm aff}$ be such that $D$ lies in the (open) Schubert cell $C_v$ attached to $v$. Fix a parametrization of $C_v$: $\bA_k^{\ell(v)} \stackrel{\sim}{\rar} C_v$, given by $a = \sum_{i=1}^{\ell(v)} a_i \unife^i \mapsto e_-(a)vI_E$ and let $D$ correspond to $a = \sum_{i} a_i \unife^i$. We compute 
\[
D \in \cB_E^{\langle \tau \rangle} \LRar e_-(a)vI_E = e_-(\tau(a))vI_E \LRar v^{-1}e_-(\tau(a) - a)v  = e_-(\unife^{ - 2 \lfloor \frac{\ell(v) + 1}{2} \rfloor} (\tau(a) - a)) \in I_E 
\]

\noindent The fact that $D$ does not lie in $\iota(\cB_F)$ (or equivalently, in $A_E$) is equivalent to $a_1 \neq 0$. By Lemma \ref{lm:trace_norm}, we have  $\ord_E(\tau(a) - a) = d+1$. Thus $D \in \cB_E^{\langle \tau \rangle}$ if and only if 
\[
(d + 1) - 2 \lfloor \frac{\ell(v) + 1}{2} \rfloor \geq 1.
\] 
\noindent Recalling that $d$ is odd, we see that this is equivalent to $\ell(v) \leq d$. Finally, note that the distance from $D$ to $\iota(\cB_F)$ is precisely $\ell(v)$. This finishes the proof.
\end{proof}


\subsection{Extended affine Deligne-Lusztig varieties}\label{sec:slight_simplification} 

Let $\breve{E}$ denote the completion of the maximal unramified extension of $E$, and let $\Sigma \subseteq \Gal(\breve{E}/F)$ be a finite set of generators of the Galois group, e.g. the set of all Frobenius lifts. 
In \cite{Ivanov_15_ram} an extended affine Deligne-Lusztig variety attached to $G$ over $F$ and a minisotropic torus split over $E$ was defined (in level $J \subseteq G(\breve{E})$) as a locally closed subset of $G(\breve{E})/J$ cut out by Deligne-Lusztig-type conditions, one for each generator $\gamma \in \Sigma$. 
For $G = \GL_2$ and $E/F$ tamely ramified quadratic extension, it turned out that one could equally define them in $G(E)/J \cap G(E)$ just by one Deligne-Lusztig condition attached to the non-trivial element of $\Gal(E/F)$ (see Remark 3.11(i) of \cite{Ivanov_15_ram}) and we conjecture that a similar fact is true whenever $S$ is split over a totally ramified Galois extension. 
At least in the present article we can and will omit the passage to the maximal unramified extension $\breve{E}/E$. More precisely, we will work the following variant of \cite{Ivanov_15_ram} Definition 2.1, which only makes sense for totally ramified extensions. 

\begin{Def} For $J \subseteq I_E$, $\dot{w} \in G(E)$, the corresponding extended affine Deligne-Lusztig variety attached to $b \in G(E)$ of level $J$ is 
\begin{equation}\label{eq:simplified_ADLV} 
X_{\dot{w}}^J(b) := \{ gJ \in G(E)/J \colon g^{-1} b \tau(g) \in J\dot{w}J \} \subseteq G(E)/J. 
\end{equation}
\end{Def}
Let $J_b(F)$ denote the $\tau$-stabilizer of $b$ in $G(E)$. If $J \subseteq I_E$ is normal and $w \in \widetilde{W}$, the group $Z(E)I_E/J$ act by $\tau$-conjugation on $\{J\dot{w}J \colon \dot{w} \in I_EwI_E \}$. Let $\widetilde{\Gamma} = \widetilde{\Gamma}_{\dot{w}} \subseteq Z(E)I_E/J$ denote the $\sigma$-stabilizer of the double coset $J\dot{w}J$ (in the general situation, the group $\Gamma$ was introduced in \cite{Ivanov_15_ram} Section 2.2.3). There is a natural action of $J_b(F) \times \widetilde{\Gamma}_{\dot{w}}$ on $X_{\dot{w}}^J(b)$ by $(g,t), xJ \mapsto gxt^{-1}J$. 

We will concentrate on the case $b = 1$ in this article. We will denote by $X_{\dot{w}}^m(1)$ the varieties in the level $I^{2m+1}_E$ and by $X_w(1)$ the varieties in the level $I_E$.


\subsection{Sketch of the situation} \label{sec:sketch_of_situation} 

In the following we mainly will deal with the $\Gamma$-torsor,
\[\
X_{\dot{w}}^m(1) \stackrel{\Gamma}{\longrightarrow} X_w(1), 
\]
\noindent where $\Gamma = \widetilde{\Gamma}_{\dot{w}} \cap I_E/I_E^{2m+1}$ is a finite group, acting by right multiplication on $X_w^m(1)$ (it is the analog of the group $T_w^F$ for which $\dot{X}_{\dot{w}} \rar X_w$ is a torsor in the classical Deligne-Lusztig theory). The varieties $X_{\dot{w}}^m(1)$ and $X_w(1)$ turn out to be discrete unions of $k$-rational points, but are not finite themselves. Nevertheless, they will naturally decompose $G(F)$-eqiuvariantly as disjoint unions:  
\[
X_{\dot{w}}^m(1) = \coprod_{g \in G(F)/I_F} g.X_{\dot{w}}^m(1)_{P_{1/2}}, \quad X_w(1) = \coprod_{g \in G(F)/I_F} g.X_w(1)_{P_{1/2}}, 
\]
\noindent where $X_{\dot{w}}^m(1)_{P_{1/2}}$ and $X_w(1)_{P_{1/2}}$ will be finite subsets and moreover, the first will be a $\Gamma$-torsor over the second:
\[
X_{\dot{w}}^m(1)_{P_{1/2}} \stackrel{\Gamma}{\longrightarrow} X_w(1)_{P_{1/2}}. 
\]
\noindent The natural $\Gamma$-action on $X_{\dot{w}}^m(1)$ extends to an action of the bigger group $\widetilde{\Gamma}$, still commuting with the $G(F)$-action. The union $\widetilde{X}_{\dot{w}}^m(1)_{P_{1/2}}$ of $\widetilde{\Gamma}$-translates of $X_{\dot{w}}^m(1)_{P_{1/2}}$, will be stable under the action of $\iota(E^{\times})I_F \subseteq G(F)$, with $\iota$ as in Section \ref{sec:group_th_data}. To any $\overline{\bQ}_{\ell}$-representation $\tilde{\theta}$ of $\widetilde{\Gamma}$, there will correspond a $G(F)$-representation
\[ 
R_{\tilde{\theta}} = \coh_c^0(X_{\dot{w}}^m(1), \overline{\bQ}_{\ell})[\tilde{\theta}] = \cInd\nolimits_{E^{\times}I_F}^{G(F)} \coh_c^0(\widetilde{X}_{\dot{w}}^m(1)_{P_{1/2}}, \overline{\bQ}_{\ell})[\tilde{\theta}].
\]
\noindent Denote the $\iota(E^{\times})I_F$-representation $\coh_c^0(\widetilde{X}_{\dot{w}}^m(1)_{P_{1/2}}, \overline{\bQ}_{\ell})[\tilde{\theta}]$ by $\Xi_{\tilde{\theta}}$. If $R_{\tilde{\theta}}$ is a cuspidal representation, then $\Xi_{\tilde{\theta}}$ is (together with implicitly determined chain order) a \emph{cuspidal type} in the sense of \cite{BushnellH_06}\S 15.8. Further, the restriction of $\Xi_{\tilde{\theta}}$ to $I_F$ coincides with $\coh_c^0(X_{\dot{w}}^m(1)_{P_{1/2}})[\chi|_{\Gamma}]$.

The plan for the rest of Section \ref{sec:extended_ADLV_and_Gamma_tilde_etc} is as follows. We will determine the varieties $X_w(1)$ and $X_{\dot{w}}^m(1)$ in Section \ref{sec:structure_of_EADLV}. Then in Sections \ref{sec:group_acting_on_right} and \ref{sec:exotic_isom} we study the group $\widetilde{\Gamma}$ under some assumptions on $w$ and $m$. Finally, in Section \ref{sec:numerical_consideration} we make a numerical consideration which suggests how to choose $w$ and $m$ (relatively to each other) such that the representations $R_{\tilde{\theta}}$ get irreducible. In Section \ref{sec:ordinary_reps_of_GL2F_wild_ram_case} we will study those irreducible $R_{\tilde{\theta}}$'s.


\subsection{Structure of some extended affine Deligne-Lusztig varieties}\label{sec:structure_of_EADLV} 

\subsubsection{Iwahori level}

Let $w = \matzz{}{\unife^{-n}}{\unife^n}{} \in W_{\rm aff}$ with $\ell(w) = 2n-1$. Let 
\[
 \dot{v} = \matzz{}{\unife^{- ( \lfloor \frac{d+1}{2} \rfloor + \lfloor \frac{n+1}{2} \rfloor)}}{\unife^{ \lfloor \frac{d+1}{2} \rfloor + \lfloor \frac{n}{2} \rfloor}}{} 
 \]
\noindent and let $v$ denote the image of $\dot{v}$ in $W_{\rm aff}$. Then $\ell(v) = n+d$. Parametrize $C_v$ by $\bA_k^{n+d} \stackrel{\sim}{\rar} C_v$, $a = \sum_{i = 1}^{n+d} a_i \unife^i \mapsto e_-(a)\dot{v}I$. Note that the locus $a_1 \neq 0$ in $C_v$ is independent from the parametrization: intrinsically it is given as the set of alcoves in $C_v$, whose vertex of departure from $\iota(\cB_F)$ is $P_{1/2}$. We denote this locus by $C_v(a_1 \neq 0)$.

\begin{prop}\label{prop:eADLV_Iwahori_level} The Iwahori-level extended affine Deligne-Lusztig variety is given by the following $G(F)$-equivariant isomorphism, 
\[
X_w(1) \cong \coprod_{g \in G(F)/I_F} g.C_v(a_1 \neq 0).
\]
\end{prop}

\begin{proof} The proof is along the lines of \cite{Ivanov_15_ram} Proposition 3.4, the only difference being that $\iota(\cB_F) \subsetneq \cB_E^{\langle \tau \rangle}$. This difference is fully controlled by Proposition \ref{prop:barbs} and is reflected in the different choice of $v$ (in \cite{Ivanov_15_ram} one had $\ell(v) = n$, whereas here we have $\ell(v) = n + d$).
\end{proof}


\subsubsection{Higher levels}
Let 
\begin{equation}\label{eq:def_of_dot_w}
\dot{w} = \matzz{}{\unife^{-n}}{\unife^n}{} \matzz{\varepsilon^{\lfloor \frac{n}{2} \rfloor}}{}{}{\varepsilon^{-\lfloor \frac{n+1}{2} \rfloor}}
\end{equation}
\noindent be a lift of $w$ to $G(E)$. For $m \geq 0$, let $C_v^m$ be the preimage of the Schubert cell $C_v$ under $G(E)/I_E^{2m+1} \tar G(E)/I_E$. A parametrization of $C_v^m$ is given by 
\begin{eqnarray}
\psi_{\dot{v}}^m \colon \fp_E/\fp_E^{n + d + m + 1} \times (U_E/U_E^{m+1})^2 \times \fp_E/\fp_E^{m+1} &\stackrel{\sim}{\longrar}& C_v^m = I_E \dot{v} I_E/I_E^{2m+1} \nonumber \\ \label{eq:explicit_Cvm_param}
(a, C,D,B) &\mapsto& e_-(a) \dot{v} e_-(B)e_0(C,D)  I_E^{2m+1}.
\end{eqnarray}
The difference between this parametrization and the one in \cite{Ivanov_15_ram} Section 3.1.7 is that the variable $a$ here is equal to $a+\unife^{n+d+1}A$ from there. This simplifies the computations. For the next proposition recall the notation $(\fp_E^a/\fp_E^b)^{\ast}$ from Section \ref{sec:basic_notation}.

\begin{prop}\label{prop:eADLV_higher_level}
Let $m \geq 0$. There is a $G(F)$-- and $\Gamma$--equivariant isomorphism
\[
X_{\dot{w}}^m(1) \cong \coprod_{G(F)/I_F} g . X_{\dot{w}}^m(1)_{P_{1/2}},
\]
\noindent where $X_{\dot{w}}^m(1)_{P_{1/2}} \subseteq C_v^m$ is given in coordinates \eqref{eq:explicit_Cvm_param} by conditions
\begin{eqnarray*}
a &\in& (\fp_E/\fp_E^{n + d + m + 1})^{\ast} \\
D &=& \varepsilon^{\lfloor \frac{d+1}{2}\rfloor} \tau(C)R^{-1} \\
B &=& \unife^n R^{-1},
\end{eqnarray*}
\noindent where 
\begin{eqnarray*}
R &:=& \unife^{-(d+1)} (\tau(a) + a) \\
\end{eqnarray*}
is a well-defined element of $U_E/U_E^{m+1}$. In particular, $X_{\dot{w}}^m(1)_{P_{1/2}}$ is just a finite discrete set of $k$-rational points. 
\end{prop}

\begin{proof} The same statement for a tamely ramified $E/F$ and with respect to another parametrization was shown in \cite{Ivanov_15_ram}  Theorem 3.8. The proof remains the same. For convenience, we sketch the proof, including the key computation (which gets significantly easier with the parametrization used here). The statement about the decomposition into a disjoint union is clear from Proposition \ref{prop:eADLV_Iwahori_level} and the fact that $X_{\dot{w}}^m(1)$ lies over $X_w(1)$. The well-definedness of $R$ follows from Lemma \ref{lm:trace_norm}(i).  We make an auxiliary computation,
\[
\dot{v}^{-1} e_-(a) e_-(\tau(a)) \tau(\dot{v}) = e_-(\unife^nR^{-1}) \dot{w} e_0(\varepsilon^{\frac{d+1}{2}} R^{-1}, \varepsilon^{-\frac{d+1}{2}}R) e_-(\varepsilon^{n+d+1}\unife^nR^{-1}).
\]
Using it we compute for $x \in C_v^m$ with coordinates $a,C,D,B$,
\begin{eqnarray*}
x^{-1} \tau(x) &\sim& e_0(C^{-1},D^{-1}) e_-(B) \dot{v}^{-1} e_-(a) e_-(\tau(a)) \tau(\dot{v}) e_-(\tau(B)) e_0(\tau(C),\tau(D)) \\
&\sim& e_0(C^{-1},D^{-1}) e_-(B + \unife^nR^{-1}) \dot{w} \dots \\
 &\quad& \dots e_0(\varepsilon^{\frac{d+1}{2}} R^{-1}, \varepsilon^{-\frac{d+1}{2}}R) e_-(\tau(B + \unife^nR^{-1} )) e_0(\tau(C),\tau(D)) \\
&\sim& e_-(CD^{-1}(B + \unife^n R^{-1})) \dot{w} e_0(\varepsilon^{\frac{d+1}{2}} R^{-1}D^{-1}\tau(C), \varepsilon^{-\frac{d+1}{2}}RC^{-1}\tau(D)) \dots \\
 &\quad& \dots e_-(\tau(CD^{-1}(B + \unife^n R^{-1}))),
\end{eqnarray*}
\noindent where we write $\sim$ to indicate that two elements belong to the same $I_E^{2m+1}$-double coset. Thus the condition $I_E^{2m+1} x^{-1}\tau(x)I_E^{2m+1} = I_E^{2m+1} \dot{w} I_E^{2m+1}$ is equivalent to the three conditions $B + \unife^n R^{-1} = 0$, $\varepsilon^{\frac{d+1}{2}} R^{-1}D^{-1}\tau(C) = 1$ and $\varepsilon^{-\frac{d+1}{2}}RC^{-1}\tau(D) = 1$. The proposition follows. 
\end{proof}

Recall the embedding $\iota$ from Section \ref{sec:group_th_data}.
\begin{cor}\label{cor:smallest_tilde_gamma_stable_subscheme}
The smallest $\widetilde{\Gamma}$--stable subscheme $\widetilde{X}_{\dot{w}}^m(1)_{P_{1/2}}$ of $X_{\dot{w}}^m(1)$, containing $X_{\dot{w}}^m(1)_{P_{1/2}}$ is $\iota(E)^{\times} I_F$-stable. There is a $G(F)$-- and $\widetilde{\Gamma}$--equivariant isomorphism
\[
X_{\dot{w}}^m(1) \cong \coprod_{G(F)/\iota(E^{\times})I_F} g . \widetilde{X}_{\dot{w}}^m(1)_{P_{1/2}},
\]
\end{cor}

\begin{proof}
As $X_{\dot{w}}^m(1)_{P_{1/2}}$ is $I_F$-stable and the left and right actions commute, it is clear that $\widetilde{X}_{\dot{w}}^m(1)_{P_{1/2}}$ is $I_F$-stable. It is enough to show that $\widetilde{X}_{\dot{w}}^m(1)_{P_{1/2}}$ is stable under the left multiplication with $\iota(\unife)$. This is a matter of a direct computation (cf. Proposition \ref{prop:all_actions}(ii))
\end{proof}

\begin{rem}\label{rem:suffices_to_look_at_special_choices}
Let $\dot{w}$ and $w$ be as in \eqref{eq:def_of_dot_w}. Let $\dot{w}_1$ be a second lift of $w$ to $G(E)$. If the cosets $I_E^{2m+1} \dot{w} I_E^{2m+1}$, $I_E^{2m+1} \dot{w}_1 I_E^{2m+1}$ are not $\tau$-conjugate by $I_E/I_E^{2m+1}$, $X_{\dot{w}_1}^m(1) = \emptyset$. If they are $\tau$-conjugate, then a conjugating element defines (by right multiplication) an isomorphism $X_{\dot{w}}^m(1) \stackrel{\sim}{\rar} X_{\dot{w}_1}^m(1)$. Thus the special choices $\dot{w}$ as in \eqref{eq:def_of_dot_w} already cover all interesting cases.
\end{rem}


\subsection{The group acting on the right} \label{sec:group_acting_on_right} 

Recall that $Z$ denotes the center of $G$. The group $\widetilde{\Gamma} \subseteq Z(E)I_E/I_E^{2m+1}$ is the stabilizer of the double coset $I_E^{2m+1} \dot{w} I_E^{2m+1}$ for the action of $Z(E)I_E/I_E^{2m+1}$ on the set of all $I_E^{2m+1}$-double cosets lying in $I_EwI_E$, given by $(i,I_E^{2m+1}\dot{w}I_E^{2m+1}) \mapsto I_E^{2m+1} i^{-1} \dot{w} \tau(i) I_E^{2m+1}$.   This is well-defined as $I_E^{2m+1}$ is normal in $I_E$. We also need the subgroup $\Gamma := \widetilde{\Gamma} \cap I_E/I_E^{2m+1}$.  For $t \in E^{\times}/U_E^{m+1}$ and $r \in \caO_E/\fp_E^m$, write

\begin{eqnarray} 
P_r &:=& 1 + \varepsilon^n\unife^{2n}r \tau(r) \in \caO_E/\fp_E^{m+1} \\ 
i(t,r) &:=& \matzz{1}{r}{}{1} \matzz{1}{}{\varepsilon^n\unife^{2n}\tau(r)P_r^{-1}}{1} \matzz{t}{}{}{\tau(t)P_r^{-1}} \in I_E/I_E^{2m+1}
\end{eqnarray}

\begin{lm}\label{lm:structure_of_Gamma}
Let $\dot{w}$ be as in \eqref{eq:def_of_dot_w}. We have
\[
\widetilde{\Gamma} = \{ i(t,r) \colon t \in E^{\times}/U_E^{m+1}, r \in \caO_E/\fp_E^m \} \subseteq Z(E)I_E/I_E^{2m+1},
\]
\noindent and $\Gamma$ is the subgroup given by the condition $t\in U_E$.
\end{lm}

\begin{proof}
The proof is a straightforward computation with $2 \times 2$-matrices.
\end{proof}

Now we study $\widetilde{\Gamma}$ under the assumptions $m \leq 2n + d - 1$ and $2n > d$. The first assumption is justified by considerations in Section \ref{sec:numerical_consideration}. The second assumption makes the structure of $\widetilde{\Gamma}$ more simple. For $x$ an element of a subquotient of $E$, we denote the scalar $2\times 2$-matrix with diagonal entries equal $x$, again by $x$. For elements $x,y$ of a group, we write $[x,y] := xyx^{-1}y^{-1}$ for their commutator.

\begin{lm}\label{lm:direct_computations} Assume $m \leq 2n + d - 1$ and $2n > d$. Let $r,u \in \caO_E/\fp_E^m$, $t \in E^{\times}/U_E^{m+1}$. Then 
\begin{itemize}
\item[(i)] We have $P_r^{-1} = \tau(P_r) = P_r = 1 + \varepsilon^n \unife^{2n} r^2$ and $P_r P_u = P_{r+u}$.
\item[(ii)]  We have 
\[
i(1,r) = P_r \matzz{1}{r}{\varepsilon^n \unife^{2n}r}{1}, \quad i(1,r)^{-1} = \matzz{1}{r}{\varepsilon^n \unife^{2n}r}{1}.
\]
\item[(iii)] The elements $i(1,r)$, $i(1,u)$ of $\Gamma$ commute.
\item[(iv)] The commutator of $i(t,0)$ and $i(1,r)$ is
\[
[i(t,0), i(1,r)] = \matzz{1}{(1 + t\tau(t)^{-1})r}{0}{1}.
\]
\end{itemize}
\end{lm}

\begin{proof}
(i): We have $P_r^2 = 1$ because $4n \geq m+1$ by assumptions, whence $P_r^{-1} = P_r$. Obviously, we also have $\tau(P_r) = P_r$. Finally, Lemma \ref{lm:trace_norm}(i) gives $r \equiv \tau(r) \mod \fp_E^d$. Hence by assumption $m \leq 2n + d - 1$, we have $P_r = 1 + \varepsilon^n \unife^{2n} r^2$. Now, 
\begin{eqnarray*} 
P_r P_u &=& (1 + \varepsilon^n\unife^{2n}r^2)(1 + \varepsilon^n\unife^{2n}u^2) \\ 
&=& 1 + \varepsilon^n\unife^{2n} (r^2 + u^2) \\
&=&  1 + \varepsilon^n\unife^{2n} (r+u)^2 = P_{r+u},
\end{eqnarray*}
\noindent where we used once again that $4n \geq m+1$. This shows (i). Part (ii) is an immediate computation using (i).

\noindent (iii): We compute
\begin{eqnarray*}
[i(1,r), i(1,u)] &=& P_{r}P_u \matzz{1 + \varepsilon^n\unife^{2n}ru}{r+u}{\varepsilon^n \unife^{2n}(r+u)}{1 + \varepsilon^n\unife^{2n} ru}^2 \\
&=& P_{r+u} P_{r+u} \\
&=& 1.
\end{eqnarray*}
where the first equation uses (ii), and the rest follows by applying (i) and $4n \geq m+1$.

\noindent (iv): Using the assumptions, the already proven parts of the lemma and Lemma \ref{lm:trace_norm}(ii), we compute
\begin{eqnarray*}
[i(t,0), i(1,r)] &=& P_r \matzz{1}{t\tau(t)^{-1}r}{t^{-1}\tau(t)\varepsilon^n \unife^{2n}r}{1} \matzz{1}{r}{\varepsilon^n \unife^{2n}r}{1} \\
&=& P_r \matzz{1 + \varepsilon^n\unife^{2n}t\tau(t)^{-1} r^2}{r(1 + t\tau(t)^{-1})}{\varepsilon^n\unife^{2n} (1 + t^{-1}\tau(t)) r}{P_r} \\
&=& \matzz{1}{(1 + t\tau(t)^{-1})r}{0}{1}. \qedhere
\end{eqnarray*}
\end{proof}

Let
\[ 
\Gamma^{\prime} := \{ i(1,r) \in \Gamma \colon r \in \fp_E^d/\fp_E^m \} \subseteq \Gamma. 
\]

\begin{prop} \label{prop:structure_of_Gamma_tilde} Assume $m \leq 2n + d - 1$ and $2n > d$. 
\begin{itemize}
\item[(i)] The commutator subgroup of $\widetilde{\Gamma}$ is $\Gamma^{\prime}$. The elements of $\widetilde{\Gamma}/\Gamma^{\prime}$ are all of the form
\[ 
i(t,\bar{r}) := \matzz{1}{r}{}{1} \matzz{1}{}{\varepsilon^n\unife^{2n}\tau(r)}{1} \matzz{t}{}{}{\tau(t)P_{\bar{r}}^{-1}} 
\] 
\noindent with $t \in E^{\times}/U_E^{m+1}, \bar{r} \in \caO_E/\fp_E^d$, where $r$ is any lift of $\bar{r}$ to $\caO_E/\fp_E^m$ and $P_{\bar{r}} := P_r$ is a well-defined element of $\caO_E/\fp_E^m$.
\end{itemize}
Let 
\[ 
(\widetilde{\Gamma}/\Gamma^{\prime})_{diag} := \{i(t,0) \colon t \in E^{\times}/U_E^{m+1} \} \subseteq \widetilde{\Gamma}/\Gamma^{\prime}
\] 
\noindent denote the subgroup of the diagonal matrices. 
\begin{itemize}
\item[(ii)] There is an isomorphism $(\widetilde{\Gamma}/\Gamma^{\prime})_{diag} \cong E^{\times}/U_E^{m+1}$, given by $i(t,0) \mapsto t$.
\item[(iii)] For $t, t^{\prime} \in E^{\times}/U_E^{m+1}$ and $\bar{r},\bar{u} \in \caO_E/\fp_E^d$ one has
\[
i(t,\bar{r}) i(t^{\prime}, \bar{u}) = i(tt^{\prime}(1 + \unife^{2n}ru), \bar{r} + \bar{u})
\]
\item[(iv)] There is an exact sequence of abelian groups
\[
0 \rar (\widetilde{\Gamma}/\Gamma^{\prime})_{diag} \rar \widetilde{\Gamma}/\Gamma^{\prime} \rar \caO_E/\fp_E^d \rar 0,
\]
\noindent where the right map is given by $i(t,\bar{r}) \mapsto \bar{r}$.
\end{itemize}
\end{prop}

\begin{proof}  
Note  that $i(t,r) = i(1,r) i(t,0)$. Hence $\Gamma$ is generated by all elements of the form $i(t,0)$ and $i(1,r)$. Thus (i) follows from the commutator formulae in Lemma \ref{lm:direct_computations} along with Lemma \ref{lm:trace_norm}(ii) (and the obvious fact that $[i(t,0), i(t^{\prime},0)] = 1$). Part (ii) of the lemma is immediate.  

(iii): We compute in $\widetilde{\Gamma}/\Gamma^{\prime}$,
\begin{eqnarray*}
i(1,\bar{r}) i(1,\bar{u}) i(1,\bar{r}+\bar{u})^{-1} &=& P_r P_u \matzz{1 + \unife^{2n}ru}{r+u}{\unife^{2n}(r+u)}{1 +  \unife^{2n}ru} \matzz{1}{r+u}{ \unife^{2n}(r+u)}{1} \\
&=& P_{r+u} \matzz{1 +  \unife^{2n}(ru + (r+u)^2)}{ \unife^{2n}ru(r+u)}{0}{1 +  \unife^{2n}(ru + (r+u)^2)} \\
&=& \matzz{1 +  \unife^{2n}ru}{ \unife^{2n}ru(r+u)}{0}{1 +  \unife^{2n}ru} \\ 
&=& \matzz{1 +  \unife^{2n}ru}{0}{0}{1 + \unife^{2n}ru} \\
&=& i(1 +  \unife^{2n}ru, 0),
\end{eqnarray*}
\noindent where we used several times that $2n \geq d$. We thus have
\begin{eqnarray*}
i(t,\bar{r}) i(t^{\prime}, \bar{u}) &=& i(1,\bar{r}) i(t, 0) i(1,\bar{u}) i(t^{\prime}, 0) \\
&=& i(1,\bar{r}) i(1,\bar{u}) i(tt^{\prime}, 0) \\
&=& i(1, \bar{r} + \bar{u}) i(tt^{\prime}(1 + \unife^{2n}ru), 0) \\
&=& i(tt^{\prime}(1 +  \unife^{2n}ru), \bar{r} + \bar{u}),
\end{eqnarray*}
\noindent where the second equality follows from commutativity of $\widetilde{\Gamma}/\Gamma^{\prime}$, which was shown in part (i), and the third equality follows from the computation above. This finishes the proof of (iii).

(iv): It is enough to check that the map $\alpha \colon \widetilde{\Gamma}/\Gamma^{\prime} \rar \caO_E/\fp_E^d$ given by $\alpha(i(t,\bar{r})) = \bar{r}$ is a homomorphism. Indeed, $\alpha$ sends the neutral element $i(1,0) \in \widetilde{\Gamma}/\Gamma^{\prime}$ to the neutral element $0 \in \caO_E/\fp_E^d$, and by part (iii), $\alpha$ respects the group laws. \qedhere
\end{proof}

For a group $A$, let $A^{\vee}$ denote the $\overline{\bQ}_{\ell}^{\times}$-valued characters of $A$. Proposition \ref{prop:structure_of_Gamma_tilde} immediately implies:

\begin{cor} \label{cor:characters_of_Gamma_tilde}
We have $\widetilde{\Gamma}^{\vee} =  (\widetilde{\Gamma}/\Gamma^{\prime})^{\vee}$. Further, there is an exact sequence
\[ 0 \rar (\caO_E/\fp_E^d)^{\vee} \rar (\widetilde{\Gamma}/\Gamma^{\prime})^{\vee} \rar (E^{\times}/U_E^{m+1})^{\vee} \rar 0. \]
\end{cor}


\subsection{An ``exotic'' isomorphism} \label{sec:exotic_isom}

Let $m \leq 2n + d - 1$ and $2n > d$. Consider the push-out in the category of abelian groups,
\begin{equation} \label{eq:def_Pi}
\Pi := E^{\times}/ U_E^{m+1} \times_{\fp_E^{n+d}/\fp_E^{m+1}} \fp_E^n/\fp_E^{m+1},
\end{equation}
\noindent with respect to the natural inclusion $\fp_E^{n+d}/\fp_E^{m+1} \har \fp_E^n/\fp_E^{m+1}$ and the map $\fp_E^{n+d}/\fp_E^{m+1} \har E^{\times}/ U_E^{m+1}$ given by $x \mapsto 1 + x$. As $\fp_E^{n+d}/\fp_E^{m+1}$ is $2$-torsion, $\Pi$ is the quotient of $E^{\times}/ U_E^{m+1} \times \fp_E^n/\fp_E^{m+1}$ by the image of the diagonal embedding of $\fp_E^{n+d}/\fp_E^{m+1}$ (the sign can be ignored).

Denote by $x \mapsto \overline{x}$ the natural projection $\caO_E/\fp_E^{n+d} \tar \caO_E/\fp_E^d$.

\begin{prop}\label{prop:beta}
There is an isomorphism depending only on the choice of the uniformizer $\unife \mod U_E^d$,
\[
\begin{aligned}
\Pi &\stackrel{\sim}{\longrar} \widetilde{\Gamma}/\Gamma^{\prime}, \\
\beta \colon (x,y) &\mapsto \left(x(1 + y)^{-1}, \overline{\unife^{-n}y}(1 + \overline{y})^{-1} \right),
\end{aligned}
\]
where the last $y$ is seen as an element of $\caO_E/\fp_E^{n+d}$ via the natural map $\fp_E^n/\fp_E^{m+1} \rar \caO_E / \fp_E^{n+d}$. 
\end{prop}

\begin{proof}
Straightforward computation.
\end{proof}

\begin{rem}\label{rem:about_the_exotic_iso} The isomorphism $\beta$ from Proposition \ref{prop:beta} is in some sense an exotic one. Indeed, at least if $n \geq d$, we have a natural isomorphism 
\[\widetilde{\Gamma}/\Gamma^{\prime} \cong E^{\times}/U_E^{m+1} \times_{U_E^{n+d}/U_E^{m+1}} U_E^n/U_E^{m+1}, \]
given by regarding $\widetilde{\Gamma}/\Gamma^{\prime}$ as the push-out of its subgroups $E^{\times}/U_E^{m+1}$ and $\{i(t,r) \colon t \equiv 1 \mod \fp_E^{n+d} \}$. Now we have $U_E^{n+d}/U_E^{m+1} \cong \fp_E^{n+d}/\fp_E^{m+1}$, simply by $1 + x \mapsto x$. On the other side $U_E^n/U_E^{m+1}$ is obviously non-isomorphic to $\fp_E^n/\fp_E^{m+1}$ if $2n < m+1$ (the second group is killed by $2$, the first has non-trivial $4$-torsion).
\end{rem}

The following lemma will be used in Section \ref{sec:irred_and_cusp}.
\begin{lm}\label{lm:kind_of_splitting} Let $m \leq 2n + d - 1$ and $2n > d$.
The map 
\[
\caO_E/\fp_E^{n+d} \rar \Gamma/\Gamma^{\prime}, \quad x \mapsto i_x := i((1 + \unife^n \varepsilon^{\frac{d+1}{2}}x)^{-1}, x(1 + \unife^n x)^{-1} \mod \fp_E^d)
\]
\noindent is a homomorphism.
\end{lm}

\begin{proof}
For $x \in \caO_E/\fp_E^{n+d}$, let $t_x := (1 + \unife^n \varepsilon^{\frac{d+1}{2}}x)^{-1}$, $\bar{r}_x := x(1 + \unife^n x)^{-1} \mod \fp_E^d$, such that $i_x = i(t_x, \bar{r}_x)$. For $x, y \in \caO_E/\fp_E^{n+d}$ we have (using Proposition \ref{prop:structure_of_Gamma_tilde}(iii))
\[
\begin{aligned}
i_x i_y &= i(t_x,\bar{r}_x) i(t_y,\bar{r}_y) = i(t_x t_y(1 + \unife^{2n}\bar{r}_x\bar{r}_y), \bar{r}_x + \bar{r}_y) \\
i_{x+y} &= i(t_{x+y}, \bar{r}_{x+y}),
\end{aligned}
\]
\noindent Using the assumptions (which in particular allow to kill $\varepsilon$ in all terms containing $\unife^{2n}$) one easily computes,
\[
t_x t_y(1 + \unife^{2n}\bar{r}_x\bar{r}_y) = t_{x+y}.
\]

A further simple calculations involving the assumptions shows $\bar{r}_x + \bar{r}_y = \bar{r}_{x+y}$. This finishes the proof of the lemma.
\end{proof}


\subsection{A numerical consideration} \label{sec:numerical_consideration} 

The varieties $X_{\dot{w}}^m(1)$ depend on two parameters, the element $\dot{w}$ and the level $m$. The essential part of the choice of $\dot{w}$ is given by the choice of its image $w$ in $\widetilde{W}$. To guarantee that $X_w(1)$ is non-empty we should choose $w \in W_{\rm aff} \subseteq \widetilde{W}$ with $\ell(w) = 2n - 1$ odd, with some $n \geq 1$ (or $w = 1$, which is the 'boring' case, giving principal series representations). Now $w$ is essentially characterized by its length, hence by the integer $n$, and we are left with the following question.

\begin{quest} 
How to choose $m,n$ such that if $w$ is of length $\ell(w) = 2n - 1$, the representation $R_{\chi} = \coh_c^0(X_{\dot{w}}^m(1), \overline{\bQ}_{\ell})[\chi]$ is an irreducible supercuspidal $G(F)$-representation for each (sufficiently generic) character $\chi$ of $\widetilde{\Gamma}$?
\end{quest} 

The same situation as described in Section \ref{sec:sketch_of_situation} was studied in \cite{Ivanov_15_ram} for $G = \GL_2$ over $F$ with $\charac F > 2$ and a totally (tamely) ramified torus in $G$. To motivate what is coming, let us recall this case first. Therefore, assume for a second that $\charac F > 2$ and let $E/F$ be a totally tamely ramified extension of degree $2$. What is said in Section \ref{sec:sketch_of_situation} about the decomposition of $X_w(1)$ and $X_{\dot{w}}^m(1)$ in a disjoint union of finite sets is also true in this situation. With $n$ as above, the cardinality of $X_w(1)_{P_{1/2}}$ is just $(q-1)q^{n-1}$. If we for a second believe in some equidistribution of the representations $R_{\chi}$ in the cohomology of $X_{\dot{w}}^m(1)$, we should have
 \[
 \dim \Xi_{\chi} = \sharp X_w(1)_{P_{1/2}} = (q-1)q^{n-1}.
 \]
\noindent On the other side, the group $\widetilde{\Gamma}$ has a natural quotient $E^{\times}/U_E^{m+1}$. Let $\chi$ be the inflation to $\widetilde{\Gamma}$ of a minimal character of $E^{\times}$ of level $m$ (this forces $m$ to be odd). Then the corresponding (by Langlands-correspondence) $G(F)$-representation will have level $\frac{m}{2}$. As we expect $R_{\chi}$ to be (up to a rectifier term) this corresponding representation, \cite{BushnellH_06} 27.6 Lemma gives a formula for the dimension of $\Xi_{\chi}$: 
\[
\dim \Xi_{\chi} = (q-1)q^{\frac{m-1}{2}}.
\]
\noindent Comparing the two formulas for the dimension of $\Xi_{\chi}$, we deduce $m = 2n - 1$, which gives the right hint how to choose $w$ and $m$, such that $R_{\chi}$ is indeed irreducible (cf. \cite{Ivanov_15_ram}, Section 4.2).

Now we come back to the situation in the present article and apply similar considerations. Proposition \ref{prop:eADLV_Iwahori_level} says that $\sharp X_w(1) = (q-1)q^{n+d-1}$, thus we may expect 
\[
\dim \Xi_{\chi} = \sharp X_w(1)_{P_{1/2}} = (q-1)q^{n+d-1}.
\]
\noindent On the other side, granted some nice relation between $\widetilde{\Gamma}$ and $E^{\times}$ (which will be studied in Section \ref{sec:group_acting_on_right}), let us assume that $\chi$ is a character of $\widetilde{\Gamma}$, which corresponds to a minimal character of $E^{\times}$. This character has level $m \geq d$, and as we want to deal with ramified representations, we may assume that $m > d$, and hence automatically, that $m$ is even (see \cite{BushnellH_06}  \S41.4). The $G(F)$-representation corresponding to this minimal character via Langlands will be of level $\frac{m+d}{2}$, and hence by \cite{BushnellH_06} 27.6 Lemma we expect
\[
\dim \Xi_{\chi} = (q-1)q^{\frac{m+d-1}{2}}.
\]
\noindent A comparison of these two formulae suggests to choose $w$ such that $m = 2n + d - 1$, which turns out to be a good choice. We will apply it in Section \ref{sec:ordinary_reps_of_GL2F_wild_ram_case}.


\section{Representations of $\GL_2(F)$ in the wildly ramified case} \label{sec:ordinary_reps_of_GL2F_wild_ram_case}

We continue to work with notation from Section \ref{sec:extended_ADLV_and_Gamma_tilde_etc}. In particular, $G = \GL_2$ and $F$ has characteristic $2$. All characters and representations considered in this section have coefficients in $\overline{\bQ}_{\ell}$.

\subsection{Preliminaries}\label{sec:prelims_on_reps_and_levels} We use the following standard notation and terminology from \cite{BushnellH_06}. For a $\overline{\bQ}_{\ell}^{\times}$-valued character $\phi$ of $F^{\times}$, $\phi_E := \phi \circ \N_{E/F}$ the corresponding character of $E^{\times}$ and by $\phi_G := \phi \circ \det$ the corresponding character of $G(F)$. The \emph{level} $\ell(\chi)$ of a multiplicative character $\chi$ of $E$ is the smallest integer $m \geq 0$, such that $\chi$ is trivial on $U_E^{m+1}$. A character $\chi$ of $E^{\times}$ is called \emph{minimal}, if $\ell(\chi) \leq \ell(\chi\phi_E)$ for all characters $\phi$ of $F^{\times}$. Similarly, the (normalized) level $\ell(\pi)$ is defined in \cite{BushnellH_06} \S 12.6 for smooth irreducible representations of $G(F)$, and such a representation is called \emph{minimal}, if $\ell(\pi) \leq \ell(\pi\phi_G)$ for all characters $\phi$ of $F^{\times}$. 

\begin{lm}[\cite{BushnellH_06}\S 41.4]
Let $\xi$ be a character of $E^{\times}$.
\begin{itemize}
 \item[(i)] If $\xi$ is minimal over $F$, then $\ell(\xi) \geq d$.
 \item[(ii)] Suppose $\ell(\xi) \geq d$; then $\xi$ is minimal over $F$ if and only if $\ell(\xi) \not\equiv d\,(\!\!\!\mod 2)$.
\end{itemize}
\end{lm}

By \cite{BushnellH_06} \S41.4 Proposition, the representation $\Ind_{E/F} \xi$ (the induction to the Weil group of $F$ from the Weil group of $E$ of the character induced on it by $\xi$ via class field theory) is unramified if $\ell(\xi) = d$. As we want construct ramified representations, and as $d$ is always odd (as we are in the equal characteristic case!) only those $\xi$ for which $\ell(\xi) > d$ is even are interesting for us.


\subsection{Definition of $R_{\tilde{\theta}}$} 

Let $m > d$ be an even integer and let $\theta$ be a character of $E^{\times}$ of level $m$ (thus minimal). Define $n \geq 1$ by 
\[
m = 2n + d - 1
\]
(this is justified by Section \ref{sec:numerical_consideration}). We will assume that $2n > d$ in the following. This restriction is given by Section \ref{sec:group_acting_on_right} and by the fact that it significantly simplifies the trace computations below. For these given $n$ and $m$, let $\dot{w}$ be as in \eqref{eq:def_of_dot_w}. We then have the corresponding group $\widetilde{\Gamma}$, studied in Section \ref{sec:group_acting_on_right} and the variety $X_{\dot{w}}^m(1)$. By Corollary \ref{cor:characters_of_Gamma_tilde} and Proposition \ref{prop:structure_of_Gamma_tilde}(i), characters of $\widetilde{\Gamma}$ which restrict to $\theta$ on $E^{\times}/U_E^{m+1}$ form a homogeneous space under $(\caO_E/\fp_E^d)^{\vee}$. For a lift $\tilde{\theta}$ of $\theta$ to a character of $\widetilde{\Gamma}$, we define the smooth $G(F)$-representation
\begin{equation}\label{eq:def_of_Rtheta}
R_{\widetilde{\theta}} := \coh_c^0(X_{\dot{w}}^m(1), \overline{\bQ}_{\ell})[\widetilde{\theta}].
\end{equation}
We will show that $G(F)$-representations are smooth irreducible and supercuspidal. Moreover, they turn out to be \emph{ordinary} in the sense of \cite{BushnellH_06} \S 44.1 Definition, that is the corresponding (under Langlands) Galois-representation is induced from an index $2$ subgroup of the Weil group of $F$ (cf. \cite{BushnellH_06} \S 41.3). 

It is highly interesting to perform the calculations below (possibly in some simplified way, which makes the traces more accessible), to determine what happens beyond the case $2n > d$, and in particular, whether also the \emph{exceptional} (i.e., not ordinary) representations of $G(F)$ are realized by $X_{\dot{w}}^m(1)$ (with $m<3d$).


\subsection{Group actions} \label{sec:group_actions}

For $g \in G(F)$, we always write $g = \matzz{g_1}{g_2}{g_3}{g_4}$. We fix a point $x = x(a,C) \in X_{\dot{w}}^m(1)_{P_{1/2}}$ with coordinates 
\[
a \in (\fp_E / \fp_E^{n + d + m + 1})^{\ast}, \quad C \in U_E/U_E^{m+1}
\] 
\noindent (cf. Proposition \ref{prop:eADLV_higher_level}). We compute the action of $I_F$ and $\Gamma$ on the coordinates of $x$. Moreover, for $h \in \iota(E^{\times})I_F$ with $\ord_F(\det(h)) = r$, we will see that $h.X_{\dot{w}}^m(1)_{P_{1/2}} = X_{\dot{w}}^m(1)_{P_{1/2}}.i(\unife^r, 0)^r$; therefore, 
\[
\beta_h \colon X_{\dot{w}}^m(1)_{P_{1/2}} \rar X_{\dot{w}}^m(1)_{P_{1/2}}, \quad yI_E^{2m+1} \mapsto hy\iota(\unife,0)^{-r}I_E^{2m+1}
\]
\noindent is an automorphism of $X_{\dot{w}}^m(1)_{P_{1/2}}$. We will also determine $\beta_h$ in terms of the coordinates.
 
\begin{prop} \label{prop:all_actions} Let $x = x(a,C)$ be a point on $X_{\dot{w}}^m(1)_{P_{1/2}}$ with coordinates $a \in (\fp_E / \fp_E^{n + d + m + 1})^{\ast}$ and $C \in U_E/U_E^{m+1}$ (in particular, $a \not\equiv 0 \mod \fp_E^2$).
\begin{itemize}
\item[(i)] Let $g \in I_F$. The action of $g$ on the coordinates $a,C$ of $x$ is given by 
\begin{eqnarray*}
g(a) &=& \frac{g_4 a + g_3}{g_2 a + g_1} \in  \fp_E / \fp_E^{n + d + m + 1} \\
g(C) &=& \frac{\det(g)C}{g_2 a + g_1}  \in U_E/U_E^{m+1}.
\end{eqnarray*}
\item[(ii)] Let $i(t,r) \in \Gamma$. The action of $i(t,r)$ on the coordinates $a,C$ of $x$ is given by
\begin{eqnarray*}
i(t,r)(a) &=& a + \unife^{n+d+1}\varepsilon^{-\frac{d+1}{2}} C\tau(C)^{-1}RH^{-1}r \in  \fp_E / \fp_E^{n + d + m + 1} \\
i(t,r)(C) &=&  C H^{-1} t \in U_E/U_E^{m+1}, 
\end{eqnarray*}
\noindent where 
\begin{equation} \label{eq:def_of_H}
H := 1 + \unife^n \varepsilon^{-\frac{d+1}{2}} C\tau(C)^{-1}r \in U_E/U_E^{m+1}.
\end{equation}
\end{itemize}
\noindent Write $a = \unife a^{\prime}$ with $a^{\prime} \in U_E/U_E^{n+d+m}$.
\begin{itemize}
\item[(iii)] Let $g \in I_F$. The action of $\beta_{g\iota(\unife)}$ on the coordinates $a,C$ of $x$ is given by
\begin{eqnarray*}
\beta_{g\iota(\unife)}(a) &=& \frac{g_4 \varepsilon \unife + g_3(a^{\prime} + 1 + \varepsilon)}{g_2 \varepsilon \unife + g_1(a^{\prime} + 1 + \varepsilon)} \in  \fp_E / \fp_E^{n + d + m + 1} \\
\beta_{g\iota(\unife),a}(C) &=& \frac{\det(g)\varepsilon C}{g_2 \varepsilon \unife + g_1(a^{\prime} + 1 + \varepsilon)} \in U_E/U_E^{m+1}.
\end{eqnarray*}
\end{itemize}
\end{prop}

\begin{proof}
Up to a change of coordinates this is shown in \cite{Ivanov_15_ram} Propositions 5.1 and 5.4. 

Part (iii) follows by combining part (i) with Lemma \ref{lm:action_of_beta_iota_unife}. \qedhere
\end{proof}
 
\begin{lm}\label{lm:action_of_beta_iota_unife} The action of $\beta_{\iota(\unife)}$ on the coordinates $a,C$ of $x$ is given by
\begin{eqnarray*}
\beta_{\iota(\unife)}(a) &=& \frac{\varepsilon \unife}{a^{\prime} + 1 + \varepsilon} \in  \fp_E / \fp_E^{n + d + m + 1} \\
\beta_{\iota(\unife)}(C) &=& \frac{\varepsilon C}{a^{\prime} + 1 + \varepsilon} \in U_E/U_E^{m+1}.
\end{eqnarray*}
\end{lm}

\begin{proof}
In the following computations, $\ast$ denotes irrelevant terms in an expression. An auxiliary computation shows:
\[
\iota(\unife) e_-(a) e_0(\varepsilon\unife, \unife)^{-1} = e_-\left(\frac{\varepsilon \unife}{a^{\prime} + 1 + \varepsilon}\right) e_0\left(\frac{a^{\prime} + 1 + \varepsilon}{\varepsilon}, \frac{\varepsilon}{a^{\prime} + 1 + \varepsilon}\right) e_+(\ast),
\]
\noindent Recall from Section \ref{sec:group_acting_on_right} that $i(\unife,0) = e_0(\unife, \epsilon\unife)$. Using the computation above we compute:
\begin{eqnarray*}
\beta_{\iota(\unife)}(x) &=& \iota(\unife) e_-(a) \dot{v} e_-(\ast) e_0(C,\ast) I_E^{2m+1} i(\unife,0)^{-1}  \\
&=& \iota(\unife) e_-(a) e_0(\varepsilon\unife, \unife)^{-1} \dot{v} e_-(\ast) e_0(C,\ast) I_E^{2m+1} \\
&=& e_-\left(\frac{\varepsilon \unife}{a^{\prime} + 1 + \varepsilon}\right) e_0\left(\frac{a^{\prime} + 1 + \varepsilon}{\varepsilon}, \frac{\varepsilon}{a^{\prime} + 1 + \varepsilon}\right) e_+(\ast) \dot{v} e_-(\ast) e_0(C,\ast) I_E^{2m+1} \\
&=& e_-\left(\frac{\varepsilon \unife}{a^{\prime} + 1 + \varepsilon}\right) e_0\left(\frac{a^{\prime} + 1 + \varepsilon}{\varepsilon}, \frac{\varepsilon}{a^{\prime} + 1 + \varepsilon}\right) \dot{v} e_-(\ast) e_0(C,\ast) I_E^{2m+1} \\
&=& e_-\left(\frac{\varepsilon \unife}{a^{\prime} + 1 + \varepsilon}\right) \dot{v} e_-(\ast) e_0\left(\frac{\varepsilon C}{a^{\prime} + 1 + \varepsilon},\ast\right) I_E^{2m+1},
\end{eqnarray*}
\noindent whence the lemma. \qedhere
\end{proof}


\subsection{Applying a trace formula} \label{sec:applying_trace_formula}

We fix a $\overline{\bQ}_{\ell}$-valued character $\widetilde{\theta}$ of $\widetilde{\Gamma}$. Recall the $\widetilde{\Gamma}$-stable subset $\widetilde{X}_{\dot{w}}^m(1) \subseteq X_{\dot{w}}^m(1)$ from Corollary \ref{cor:smallest_tilde_gamma_stable_subscheme}. As in \cite{Ivanov_15_unram} Lemma 4.5 we have:

\begin{lm} \label{lm:reduction_to_finite_space} The natural inclusion $X_{\dot{w}}^m(1)_{P_{1/2}} \har \widetilde{X}_{\dot{w}}^m(1)$ induces an isomorphism
\[ \coh_c^0(\widetilde{X}_{\dot{w}}^m(1)_{P_{1/2}})[\widetilde{\theta}] \cong \coh_c^0(X_{\dot{w}}^m(1)_{P_{1/2}})[\widetilde{\theta}|_{\Gamma}].  \]
\end{lm}

We need some notation. We write 
\[
\begin{aligned}
V_{\widetilde{\theta}} &:= \coh_c^0(\widetilde{X}_{\dot{w}}^m(1))[\widetilde{\theta}] \\
\Xi_{\widetilde{\theta}} &:= \text{the $\iota(E^{\times})I_F$-representation in $V_{\widetilde{\theta}}$}.
\end{aligned}
\]
\noindent Note that $V_{\widetilde{\theta}}$ is a finite-dimensional $\overline{\bQ}_{\ell}$-vector space.

%

Let $g \in \iota(E^{\times})I_F$. Note that $\beta_g$ (introduced in Section \ref{sec:group_actions}) induces an automorphism $\beta_g^{\ast}$ of $V_{\widetilde{\theta}}$. We write 
\begin{equation}\label{eq:def_A_g}
A_g := \{ a \in (\fp_E/ \fp_E^{n + d + m + 1})^{\ast} \colon \beta_g(a) \equiv a \mod \fp_E^{n+d+1} \}
\end{equation}
For $a \in A_g$, we write 
\[
h(g,a) := \unife^{-(n+d+1)}(\beta_g(a) - a).
\]
\noindent This is a well-defined element of $\caO_E/\fp^m$. Further, for $a \in (\fp_E / \fp_E^{n + d + m + 1})^{\ast}$, we write
\[
\beta_{g,a}^{\prime} := C^{-1} \beta_{g,a}(C) \in U_E/U_E^{m+1}, 
\]
which depends on $g$ and $a$, but not on $C$.

\begin{prop}\label{prop:traces_in_general}
Let $g \in \iota(E^{\times}) I_F$. Then 
\[
\tr(g; \Xi_{\widetilde{\theta}}) = \frac{1}{q^m} \widetilde{\theta}(i(\unife,0))^{\ord_F(\det(g))} \cdot \sum_{a \in A_g} \widetilde{\theta}(i(t_a,\bar{r}_a)),
\]
\noindent where $i(t_a,\bar{r}_a) \in \Gamma/\Gamma^{\prime}$ is given by 
\[
\begin{aligned}
\bar{r}_a &:= R^{-1} h(g,a) (1 + \unife^n h(g,a) R^{-1})^{-1} \in \caO_E/ \fp_E^d \\
t_a &:= \beta_{g,a}^{\prime}  (1 + \unife^n h(g,a) R^{-1})^{-1} \in U_E/U_E^{m+1},
\end{aligned}
\]
with $R$ as in Proposition \ref{prop:eADLV_higher_level}.
\end{prop}

\begin{proof}
As $\iota(\unife,0)$ acts in $V_{\widetilde{\theta}}$ as multiplication by the scalar $\widetilde{\theta}(i(\unife,0))$, we deduce from \cite{Boyarchenko_12} Lemma 2.12,
\begin{equation}\label{eq:traces_reduction_to_finite}
\begin{aligned}
 \tr(g; \Xi_{\widetilde{\theta}}) &= \widetilde{\theta}(i(\unife,0))^{\ord_F(\det(g))} \tr(\beta_g^{\ast}; V_{\widetilde{\theta}}) \\
&=  \frac{1}{\sharp \Gamma}\widetilde{\theta}(i(\unife,0))^{\ord_F(\det(g))} \cdot \sum_{i(t,r) \in \Gamma} \sharp S_{g,i(t,r)} \widetilde{\theta}(i(t,r)),
\end{aligned}
\end{equation}
\noindent where $S_{g,i(t,r)}$ denote the set of all $x \in X_{\dot{w}}^m(1)_{P_{1/2}}$, such that $\beta_g(x) = x.i(t,r)$. 

According to Proposition \ref{prop:all_actions}(ii) the set $S_{g, i(t,r)}$ is equal to the set of all solutions in the variables $a \in (\fp_E/ \fp_E^{n + d + m + 1})^{\ast}$, $C \in U_E/U_E^{m+1}$ of the equations
\begin{equation}\label{eq:original_two_equations}
\begin{aligned}
\beta_g(a) &\equiv a + \unife^{n+d+1}C\tau(C)^{-1} R H^{-1} r \mod \fp_E^{n+d+m+1} \\
\beta_{g,a}^{\prime} &\equiv H^{-1} t \mod \fp_E^{m+1}
\end{aligned}
\end{equation}

\noindent Any solution must satisfy $a \in A_g$, so we may assume this. The first equation may thus be rewritten as
\begin{equation}\label{eq:first_eq_mod_m_1}
r \equiv \tau(C)C^{-1} R^{-1}h(g,a)H \mod \fp_E^m.
\end{equation}
\noindent Inserting \eqref{eq:first_eq_mod_m_1} and \eqref{eq:def_of_H} alternatingly into \eqref{eq:def_of_H} and iterating this process (use that $n > 0$), we may rewrite \eqref{eq:def_of_H} as 
\begin{equation}\label{eq:better_eq_for_H-1}
H^{-1} \equiv 1 + \unife^n h(g,a) R^{-1} \mod \fp_E^{m+1}.
\end{equation}

\noindent Inserting this into \eqref{eq:original_two_equations} (we use \eqref{eq:first_eq_mod_m_1} instead of the first equation of \eqref{eq:original_two_equations}), we get rid of $H$, and our equations get equivalent to  
\begin{equation}\label{eq:new_form_of_eqs}
\begin{aligned}
r &\equiv \tau(C)C^{-1} R^{-1} h(g,a) (1 + \unife^n h(g,a) R^{-1})^{-1} \mod \fp_E^m \\
t &\equiv \beta_{g,a}^{\prime}  (1 + \unife^n h(g,a) R^{-1})^{-1} \mod \fp_E^{m+1}
\end{aligned}
\end{equation}

\noindent Now consider the equations 
\begin{equation}\label{eq:new_form_of_eqs_red}
\begin{aligned}
\bar{r} &\equiv R^{-1} h(g,a) (1 + \unife^n h(g,a) R^{-1})^{-1} \mod \fp_E^d \\
t &\equiv \beta_{g,a}^{\prime}  (1 + \unife^n h(g,a) R^{-1})^{-1} \mod \fp_E^{m+1},
\end{aligned}
\end{equation}

\noindent obtained from \eqref{eq:new_form_of_eqs} by reducing the first equation modulo $\fp_E^d$ (and using Lemma \ref{lm:trace_norm}(ii)). These equations are attached to an element $i(t,\bar{r}) \in \Gamma/\Gamma^{\prime}$. The set of all solutions in $a \in A_g$, $C \in U_E/U_E^{m+1}$ of \eqref{eq:new_form_of_eqs_red} is equal to the union of sets $S_{g, i(t,r)}$ for $i(t,r)$ varying over all preimages of $i(t,\bar{r})$ in $\Gamma$. As $\widetilde{\theta}$ factors through $\Gamma/\Gamma^{\prime}$, and as \eqref{eq:new_form_of_eqs_red} does not depend on $C$, we deduce
\[
\tr(g; \Xi_{\widetilde{\theta}}) = \frac{1}{q^m} \widetilde{\theta}(i(\unife,0))^{\ord_F(\det(g))} \cdot \sum_{i(t,\bar{r}) \in \Gamma/\Gamma^{\prime}} \sharp S_{g,i(t,\bar{r})} \widetilde{\theta}(i(t,r)),
\]

\noindent where $S_{i(t,\bar{r})}$ is the set of solutions in the variable $a \in A_g$ of the equations \eqref{eq:new_form_of_eqs_red} (the variable $C$ being eliminated). Now the proposition follows, as (by looking at equations \eqref{eq:new_form_of_eqs_red}) each $a \in A_g$ produces exactly one element $i(t_a,\bar{r}_a)$, such that $S_{i(t_a,\bar{r}_a)} = \{a\}$. \qedhere
\end{proof}

\begin{cor} \label{cor:central_char_kernel_dim} The central character of $R_{\widetilde{\theta}}$ is $\theta|_{F^{\times}}$, $\Xi_{\widetilde{\theta}}$ is trivial on $I_F^{m+d+1}$ and the space $V_{\widetilde{\theta}}$ has dimension $(q-1)q^{n+d-1}$.
\end{cor}
\begin{proof} The action of the central elements of $G(F)$ by left and right multiplication coincide, and the subgroup of $\widetilde{\Gamma}$ consisting of the diagonal matrices with entries in $F^{\times}$ acts in $R_{\widetilde{\theta}}$ via the character $\theta$. For $g \in I_F^{m+d+1}$, one easily checks that $h(g,a) \equiv 0 \mod \fp_E^{n+d}$, and the statement follows by applying Proposition \ref{prop:traces_in_general}.
\end{proof}


\subsection{Properties of $R_{\widetilde{\theta}}$}\label{sec:irred_and_cusp} 

Recall that $\theta = \tilde{\theta}|_{(\widetilde{\Gamma}/\Gamma^{\prime})_{diag}}$ is minimal of level $m$, in particular, its restriction to $U_E^m/U_E^{m+1}$ is non-trivial. For convenience, we introduce the following notation
\[
\delta := \left\lfloor \frac{n+1}{2} \right\rfloor - \left\lfloor \frac{n}{2} \right\rfloor,
\]

\begin{prop}\label{prop:unipotent_traces}
Let $g = e_-(u) \in I_F$ with $u \in \fp_F$. Then
\[
\tr(g; \Xi_{\widetilde{\theta}}) =
\begin{cases}
 0 & \text{if $\ord_F(u) < n+d$} \\ 
 -q^{n+d-1} & \text{if $\ord_F(u) = n+d$}\\
 (q-1)q^{n+d-1} &\text{if $\ord_F(u) \geq n+d+1$}
\end{cases}
\]
\end{prop}

\begin{proof}
We apply Proposition \ref{prop:traces_in_general} and use the notations from there. First we show the following simple lemma. 

\begin{lm} \label{lm:comp_of_h_for_unipotents} 
Let $g = e_-(u) \in I_F$ with $u \in \fp_F$. Then the following are equivalent
\begin{itemize}
\item[(i)] $A_g \neq \emptyset$
\item[(ii)] $A_g = (\fp_E/\fp_E^{n+d+m+1})^{\ast}$
\item[(iii)] $\ord_F(u) \geq \lfloor \frac{n+1}{2} \rfloor + \frac{d+1}{2}$.
\end{itemize} 
\noindent If these conditions hold and if we write $u = \uniff^{\lfloor \frac{n+1}{2} \rfloor + \frac{d+1}{2} + \alpha} u_0$ with $u_0 \in U_F$, then 
\[ h(g,a) = \unife^{\delta + 2\alpha} \varepsilon^{\lfloor \frac{n+1}{2} \rfloor + \frac{d+1}{2} + \alpha} u_0.
\]
\end{lm}
\begin{proof} We have $g.a = a + u$, i.e. $A_g \neq \emptyset \LRar u \equiv 0 \mod \fp_E^{n+d+1} \LRar u \equiv 0 \mod \fp_F^{\lfloor \frac{n+1}{2} \rfloor + \frac{d+1}{2}}$. If this holds, then $A_g = (\fp_E/\fp_E^{n+d+m+1})^{\ast}$. The last statement is clear by definition of $h(g,a)$.
\end{proof}

We continue with the proof of Proposition \ref{prop:unipotent_traces}. If $\ord_F(u) < \lfloor \frac{n+1}{2} \rfloor + \frac{d+1}{2}$, then $A_g = \emptyset$ by Lemma \ref{lm:comp_of_h_for_unipotents} and the statement is immediate. Thus we may assume $\ord_F(u) = \lfloor \frac{n+1}{2} \rfloor + \frac{d+1}{2} + \alpha$ with $\alpha \geq 0$, and in particular, $A_g = (\fp_E/\fp_E^{n+d+m+1})^{\ast}$. 

Assume that $\alpha \geq \lfloor \frac{n}{2} \rfloor + \frac{d+1}{2}$ (i.e., $\ord_F(u) \geq n+d+1$). Then $\delta + 2\alpha \geq n+d+1$, and applying Lemma \ref{lm:comp_of_h_for_unipotents} shows that for each $a \in A_g$, $t_a = 1 \in U_E/U_E^{m+1}$ and $\bar{r}_a  = 0 \in \caO_E/\fp^d$, which shows that $a \mapsto i(t_a, \bar{r}_a)$ is the constant map sending all of $A_g$ to the neutral element $i(1,0) \in \Gamma/\Gamma^{\prime}$. As $\sharp A_g = (q-1)q^{n+d+m-1}$, the statement follows also in this case.

It remains to deal with the case $0 \leq \alpha \leq \lfloor \frac{n}{2} \rfloor + \frac{d+1}{2} - 1$. Set 

\[
\begin{aligned}
B_{\alpha} &:= \left\{ x \in \fp_E^{\delta + 2\alpha}/\fp_E^{n+d} \colon \tau(x) = \varepsilon^{-n}x \right\} \\
B_{\alpha}^{\ast} &:= B_{\alpha} \cap \left(\fp_E^{\delta + 2\alpha}/\fp_E^{n+d}\right)^{\ast}.
\end{aligned}
\]

\begin{lm}\label{lm:technical_lm_for_unipotent_traces} The assignment $a \mapsto h(g,a)R^{-1}$ induces a map
\[
(\fp_E/\fp_E^{n+d+m+1})^{\ast} \tar B_{\alpha}^{\ast},
\]
\noindent with all fibers of cardinality $q^{n+d+m - \lfloor \frac{n+d+1}{2} \rfloor + \alpha}$. 
\end{lm}
\begin{proof}
Applying Lemma \ref{lm:trace_norm}(i) several times shows that the trace of $E/F$ composed with multiplication by $\uniff^{-\frac{d+1}{2}}$ induces a surjective map 
\[ 
\uniff^{-\frac{d+1}{2}} \overline{\Tr}_{E/F} \colon \left(\fp_E/\fp_E^{n+d+m+1}\right)^{\ast} \tar \left(\caO_F/\fp_F^{\lfloor \frac{n+d+1}{2} \rfloor}\right)^{\ast} = U_F/U_F^{\lfloor \frac{n+d+1}{2} \rfloor},
\]
\noindent which is the restriction of a homomorphism on the same spaces without $\ast$'s. In particular, all fibers have the same cardinality, equal to $q^{n+d+m - \lfloor \frac{n+d+1}{2} \rfloor}$. Multiplying the map $\uniff^{-\frac{d+1}{2}} \overline{\Tr}_{E/F}$ by the invertible factor $\varepsilon^{\frac{d+1}{2}}$ and then inverting, we obtain the map (induced by) $a \mapsto R^{-1} = (\unife^{-(d+1)}(a + \tau(a)))^{-1}$. More precisely, we have $\caO_F/\fp_F^{\lfloor \frac{n+d+1}{2} \rfloor} \subseteq \caO_E/\fp_E^{n+d}$ and $a \mapsto R^{-1}$ induces  
\[
\left(\fp_E/\fp_E^{n+d+m+1}\right)^{\ast} \tar \varepsilon^{-\frac{d+1}{2}} \left(\caO_F/\fp_F^{\lfloor \frac{n+d+1}{2} \rfloor}\right)^{\ast} \subseteq \varepsilon^{-\frac{d+1}{2}} \left(\caO_F/\fp_F^{\lfloor \frac{n+d+1}{2} \rfloor}\right) \subseteq \caO_E/\fp_E^{n+d},
\]
\noindent with (non-empty) fibers still of cardinality $q^{n+d+m - \lfloor \frac{n+d+1}{2} \rfloor}$. (Note that the map $a \mapsto R^{-1}$ is neither additive, nor multiplicative). Observe that 
\[ 
\varepsilon^{-\frac{d+1}{2}} \left(\caO_F/\fp_F^{\lfloor \frac{n+d+1}{2} \rfloor}\right) = \left\{ x \in \caO_E/\fp_E^{n+d} \colon \tau(x) = \varepsilon^{\frac{d+1}{2}}x \right\} \subseteq \caO_E/\fp_E^{n+d}.
\]
Multiplication by $\varepsilon^{\lfloor \frac{n+1}{2} \rfloor + \frac{d+1}{2} + \alpha} u_0$ maps this subgroup isomorphically onto 
\[
\varepsilon^{\lfloor \frac{n+1}{2} \rfloor + \alpha}\left(\caO_F/\fp_F^{\lfloor \frac{n+d+1}{2} \rfloor}\right) = \left\{ x \in \caO_E/\fp_E^{n+d} \colon \tau(x) = \varepsilon^{- 2 \lfloor \frac{n+1}{2} \rfloor - 2\alpha}x \right\} \subseteq \caO_E/\fp_E^{n+d}
\]
\noindent (as $u_0 \in \caO_F$), preserving the ${}^{\ast}$-subsets.  Now, multiplication by $\unife^{\delta + 2\alpha}$ maps this surjectively onto 
\[
B_{\alpha} \subseteq \fp_E^{\delta + 2\alpha}/\fp_E^{n+d},
\]
\noindent preserving the ${}^{\ast}$-subspaces. Moreover, the fibers all have cardinality $q^{\alpha}$ (being equal to the cardinality of the multiplication-by-$\unife^{\delta + 2\alpha}$ map $\caO_F/\fp^{\lfloor \frac{n+d+1}{2} \rfloor} \rar \fp_E^{\delta + 2\alpha}/\fp_E^{n+d}$). Putting all this together, we see that $a \mapsto R^{-1}h(g,a)$ in fact induces a map 
\[
\left(\fp_E/\fp_E^{n+d+m+1}\right)^{\ast} \tar B_{\alpha}^{\ast},
\]
\noindent whose fibers all have the same cardinality, equal to $q^{n+d+m - \lfloor \frac{n+d+1}{2} \rfloor + \alpha}$. This finishes the proof of the lemma. \qedhere
\end{proof}

Using Lemma \ref{lm:technical_lm_for_unipotent_traces} we may replace $a$ in the formulae in Proposition \ref{prop:traces_in_general} by $x := R^{-1}h(g,a)$. More precisely, we have
\[ 
\tr(g; \Xi_{\chi}) = q^{n+d - \lfloor \frac{n+d+1}{2} \rfloor + \alpha} \cdot \sum_{x \in B_{\alpha}^{\ast}} \widetilde{\theta}(i(t_x,\bar{r}_x)),
\]
\noindent where $i(t_x,\bar{r}_x) \in \Gamma/\Gamma^{\prime}$ is given by 
\[
\begin{aligned}
\bar{r}_x &:= x(1 + \unife^n x)^{-1} \in \caO_E/ \fp_E^d \\
t_x &:= (1 + \unife^n \varepsilon^{-\frac{d+1}{2}} x)^{-1} \in U_E/U_E^{m+1},
\end{aligned}
\]
\noindent Define $\widetilde{\theta}^{\prime}$ on $B_{\alpha}$ by setting $\widetilde{\theta}^{\prime}(x) := \widetilde{\theta}(i(t_x,\bar{r}_x))$. By Lemma \ref{lm:kind_of_splitting}, $\widetilde{\theta}^{\prime}$ is a character of $B_{\alpha}$. Moreover, $\widetilde{\theta}^{\prime}$  is non-trivial: indeed, $B_{\alpha}$ contains $B_{\lfloor \frac{n}{2} \rfloor + \frac{d+1}{2} - 1} = \fp_E^{n+d-1}/\fp_E^{n+d}$ (the condition $\tau(x) = \varepsilon^{-n}x$ gets empty here) and when $x = \unife^{n+d-1}x_0$ runs through $B_{\lfloor \frac{n}{2} \rfloor + \frac{d+1}{2} - 1}$, $i(t_x,\bar{r}_x) = i((1 + \unife^m x_0), 0)$ runs through $U_E^m/U_E^{m+1} \subseteq E^{\times}/U_E^{m+1} = (\widetilde{\Gamma}/\Gamma^{\prime})_{diag}$, and by assumption, $\widetilde{\theta}$ is non-trivial there. 

We thus have
\[ 
\tr(g; \Xi_{\chi}) = q^{n+d - \lfloor \frac{n+d+1}{2} \rfloor + \alpha} \cdot \sum_{x \in B_{\alpha}^{\ast}} \widetilde{\theta}^{\prime}(x),
\]
\noindent Observing that $B_{\alpha}^{\ast} = B_{\alpha} \sm B_{\alpha + 1}$, we deduce $\tr(g;\Xi_{\alpha}) = 0$ for $0 \leq \alpha < \lfloor \frac{n}{2} \rfloor + \frac{d+1}{2} - 1$. For $\alpha = \lfloor \frac{n}{2} \rfloor + \frac{d+1}{2} - 1$, we have
\[ \tr(g;\Xi_{\alpha}) = q^{n+d - \lfloor \frac{n+d+1}{2} \rfloor + (\lfloor \frac{n}{2} \rfloor + \frac{d+1}{2} - 1)} \cdot \sum_{x \in B_{\lfloor \frac{n}{2} \rfloor + \frac{d+1}{2} - 1}^{\ast}} \widetilde{\theta}(x) = - q^{n + d - 1}. \qedhere \]
\end{proof}

For $\alpha \geq 1$, let $N^{\alpha}$ be the subgroup of $I_F$ consisting of all lower triangular matrices with $1$'s on the main diagonal, such that the lower entry has valuation $\geq \alpha$.  Let $B$ be the Borel subgroup of lower triangular matrices of $G$.
 
\begin{cor}\label{cor:Xi_theta_as_N1-rep}
As $N^1$-representation, $\Xi_{\widetilde{\theta}}$ is the direct sum over all characters of $N^1$, which are trivial on $N^{n+d+1}$ and non-trivial on $N^{n+d}$. Moreover, $\Xi_{\widetilde{\theta}}$ is an irreducible $B(F) \cap I_F$ representation.
\end{cor}
\begin{proof} 
The first statement immediately follows from Proposition \ref{prop:unipotent_traces}. The second follows from the first as in \cite{Ivanov_15_unram} Corollary 4.12.
\end{proof}

\begin{cor}\label{cor:irred_level_minimality}
The representation $R_{\widetilde{\theta}}$ is irreducible, cuspidal and admissible. It contains a ramified simple stratum and is, in particular, ramified. Its level is $\ell(R_{\widetilde{\theta}}) = \frac{m+d}{2}$. For any character $\phi$ of $F^{\times}$, one has $0 < \ell(R_{\widetilde{\theta}}) \leq \ell(\phi R_{\widetilde{\theta}})$.
\end{cor}

\begin{proof} 
This follows from Corollary \ref{cor:Xi_theta_as_N1-rep} and \cite{Ivanov_15_ram} Proposition 4.24. 
\end{proof}


\subsection{Wild cuspidal types} \label{sec:comparison_BKH} 

Here we briefly recall the method of \cite{BushnellH_06} to produce smooth irreducible cuspidal $\overline{\mathbb{Q}}_{\ell}$-representations of $G(F)$ from certain characters of open subgroups of $G(F)$, which are compact modulo center. For definitions and general results on cuspidal types (for $G = \GL_2$) we refer to \cite{BushnellH_06}. We concentrate on the special case when the residue characteristic of $F$ is $2$. Let $\fI_F$ be the $\caO_F$-subalgebra of $\fM := {\rm Mat}_{2 \times 2}(F)$ with filtration by $\fI^r_F$ given by
\[
 \fI_F^r := \iota(\pi)^r \fI_F = \matzz{\fp_F^{\lfloor \frac{r+1}{2} \rfloor}}{\fp_F^{\lfloor \frac{r}{2} \rfloor}}{\fp_F^{\lfloor \frac{r}{2} \rfloor + 1}}{\fp_F^{\lfloor \frac{r+1}{2} \rfloor}} \subseteq \fI_F := \matzz{\caO_F}{\caO_F}{\fp_F}{\caO_F} 
\]
Then $\fI_F^{\times} = I_F$ and $I_F^r = 1 + \fI_F^r$ for $r \geq 1$.  

Fix once for all a $\overline{\bQ}_{\ell}^{\times}$-valued character $\psi$ of $F$ of level $1$ (i.e., trivial on $\fp_F$, non-trivial on $\caO_F$). Let $\psi_{\fM} := \psi \circ \tr_{\fM}$, where $\tr_{\fM}$ is the trace. Analogously, put $\psi_E := \psi \circ \tr_{E/F}$. Note that for integers $k \leq r$, $\iota \colon E \har \fM$ induces an inclusion $\fp_E^k/\fp_E^r \har \fI_F^k/\fI_F^r$. 

\begin{lm}\label{lm:dual_spaces_of_characters}
\begin{itemize}
\item[(i)] (\cite{BushnellH_06} 12.5 Proposition) Let $0 \leq k < r \leq 2k+1$ be integers. There is an isomorphism
\[
\fJ_F^{-r}/\fJ_F^{-k} \stackrel{\sim}{\longrar} (I_F^{k+1}/I_F^{r+1})^{\vee}, \qquad a + \fJ_F^{-k} \mapsto \psi_{\fM,a}|_{U_{\fJ}^{k+1}}
\]
\noindent where $\psi_{\fM,a}$ denotes the function $x \mapsto \psi_{\fM}(a(x-1))$. 
\item[(ii)] Let $0 \leq k < r \leq 2k+1$ be integers. There is an isomorphism
\[
\fp_E^{-(r+d)}/\fp_E^{-(k+d)} \stackrel{\sim}{\longrar} (U_E^{k+1}/U_E^{r+1})^{\vee} , \qquad a + \varpi^{-k}\fJ \mapsto \psi_{E,a}|_{U_{\fJ}^{k+1}}
\]
where $\psi_{E,a}$ denotes the function $x \mapsto \psi_E(a(x-1))$.
\item[(iii)] Let $k,r$ be positive integers satisfying $k+d < r \leq 2k + d + 1$. Then there is a commutative diagram 

\centerline{
\begin{xy}\label{diag:character_isos_diag}
\xymatrix{
\fp_E^{-(r+d)}/\fp_E^{-(k+d)} \ar@{->>}[r]  \ar@{^{(}->}[d] & \fp_E^{-(r+d)}/\fp_E^{-(k+2d)}  \ar[r]^{\, \sim} & (U_E^{k+d+1}/U_E^{r+1})^{\vee} \ar@{^{(}->}[d] \\
\fJ_F^{-(r+d)}/\fJ_F^{-(k+d)} \ar[r]^{\sim\,\,\,} & (I_F^{k+d+1}/I_F^{r+d+1})^{\vee} \ar@{->>}[r] & (U_E^{k+d+1}/U_E^{r+d+1})^{\vee}
}
\end{xy}
}
\noindent where the two horizontal isomorphisms are from parts (i) and (ii) of the lemma and all other maps are either induced by $\iota$ or by the natural projections.
\end{itemize}
\end{lm}

\begin{proof}
Part (ii) is immediate (the shift by $d$ coming from the discriminant of $E/F$). Part (iii) is immediate from (i) and (ii) and $\tr_{E/F} = \tr_{\fM} \circ \iota$.
\end{proof}

We abuse the notation $\psi_{E,\alpha}$, by using it for both, the (additive) character of a subquotient of $\caO_E$ and the (multiplicative) character of a subquotient of $U_E$. It will be always clear from the context, which character is meant.  

The following construction uses Lemma \ref{lm:dual_spaces_of_characters}(iii) with $r = m$ and $k = n-1$. To give a \textit{cuspidal type} $(\fJ_F, \iota(E^{\times})I_F^{n+d}, \Lambda)$, such that $\Lambda$ is trivial on $I_F^{m+d+1} = I_F^{2(n+d)}$, and such that the restriction of $\Lambda$ to $\iota(E^{\times})$ is the character $\theta \circ \iota^{-1}$, is the same as to give an element $\alpha \in \fp_E^{-(m+d)}/\fp_E^{-(n+d-1)}$, such that the restriction of $\psi_{\fM,\iota(\alpha)}$ to $U_E^{n+d}/U_E^{m+d+1}$, which factors through $U_E^{n+d}/U_E^{m+1}$ (by Lemma \ref{lm:dual_spaces_of_characters}(iii)) is equal to the restriction of $\theta$ to this subgroup, that is
\begin{equation}\label{eq:compat_alpha_psi_theta}
\psi_{E,\alpha}(x) = \theta(1+x) \quad \text{for all $x \in \fp_E^{n+d}/\fp_E^{m+1}$}.
\end{equation}

The following lemma is immediate.

\begin{lm}\label{lm:char_of_Pi_from_theta_psi_alpha}
Let $\theta$, $\psi$, $\alpha$ be as above, satisfying \eqref{eq:compat_alpha_psi_theta}. There is a unique character of the group $\Pi$ from \eqref{eq:def_Pi}, whose restriction to $E^{\times}/U_E^{m+1}$ (resp. $\fp_E^n/\fp_E^{m+1}$) is $\theta$ (resp. $\psi_{E,\alpha}$).
\end{lm}

\begin{Def}
We set 
\begin{itemize}
\item[$\bullet$] $\Lambda_{\theta, \psi, \alpha} := $ the character of $\iota(E^{\times})I_F^{n+d}$ corresponding to $\psi$, $\alpha$ and $\theta$ as above,
\item[$\bullet$] ${\rm BH}_{\theta,\psi,\alpha} := \cInd_{\iota(E^{\times})I_F^{n+d}}^{G(F)} \Lambda_{\theta, \psi, \alpha}$.
\item[$\bullet$] $(\theta, \psi_{E,\alpha}) := $ the character of $\Pi$ attached to $\theta, \psi, \alpha$ by Lemma \ref{lm:char_of_Pi_from_theta_psi_alpha}. 
\end{itemize}
\end{Def}

\begin{thm}[\cite{BushnellH_06} \S15.5 Corollary]\label{thm:every_rep_cont_cusp_type}
The map 
\[ (\fA,J,\Lambda) \mapsto \cInd\nolimits_J^{G(F)} \Lambda \]
induces a bijection between the set of conjugacy classes of all (i.e., not necessarily those considered above) cuspidal types in $G(F)$ and equivalence classes of irreducible supercuspidal representations of $G(F)$.
\end{thm}

\begin{cor}[cf. \cite{BushnellH_06} \S15.3 Theorem]
The representation ${\rm BH}_{\theta,\psi,\alpha}$ is irreducible and supercuspidal. 
\end{cor}


\subsection{Relation between geometric and type-theoretical constructions}
By Lemma \ref{lm:dual_spaces_of_characters}(iii) there are precisely $q^d$ elements $\alpha \in \fp_E^{-(m+d)}/\fp_E^{-(n+d-1)}$ satisfying \eqref{eq:compat_alpha_psi_theta}, each giving rise to the cuspidal inducing datum $(\fJ_F, \Ind\nolimits_{\iota(E^{\times})I_F^{n+d}}^{\iota(E^{\times})I_F} \Lambda_{\theta, \psi, \alpha})$, which is, in a sense, attached to $\theta$. On the geometric side there are precisely $q^d$ lifts $\tilde{\theta}$ of $\theta$ to a character of $\widetilde{\Gamma}/\Gamma^{\prime}$, each giving rise to the cuspidal inducing datum $(\fJ_F, \Xi_{\tilde{\theta}})$. Our main result is the following theorem, which states that the relation between the two families of corresponding $G(F)$-representations, $R_{\widetilde{\theta}}$ and ${\rm BH}_{\theta, \psi, \alpha}$, is naturally encoded in the dual $\beta^{\vee}$ of the isomorphism $\beta \colon \Pi \stackrel{\sim}{\rar} \widetilde{\Gamma}/\Gamma^{\prime}$ from Lemma \ref{prop:beta}. 

\begin{thm}\label{thm:relation_ADLV_BH}
Let $\tilde{\theta}$ be a character of $\widetilde{\Gamma}/\Gamma^{\prime}$ with restriction $\theta$ to $E^{\times}/U_E^{m+1}$ of level $m$. Let $\psi,\alpha$ be such that $\beta^{\vee}(\tilde{\theta}) = (\theta, \psi_{E,\alpha})$. Then
\[
R_{\widetilde{\theta}} \cong {\rm BH}_{\theta, \psi, \alpha}.
\]
\end{thm}

\begin{proof}
We have $R_{\widetilde{\theta}} \cong \cInd_{\iota(E^{\times})I_F}^{G(F)} \Xi_{\widetilde{\theta}}$, so it suffices to show that 
\begin{equation} \label{eq:need_to_show_iso_cusp_types} 
\Xi_{\widetilde{\theta}} \cong \cInd\nolimits_{\iota(E^{\times})I_F^{n+d}}^{\iota(E^{\times})I_F} \Lambda_{\theta, \psi, \alpha}.
\end{equation}
From Corollaries \ref{cor:central_char_kernel_dim}, \ref{cor:irred_level_minimality} and Lemma \ref{lm:level_and_minimality_of_BH_side} it follows that both sides are cuspidal inducing data sharing and
\begin{itemize}
 \item same underlying subgroup $\iota(E^{\times}) I_F$,
 \item same central character $\theta|_{F^{\times}}$,
 \item same level $\frac{m+d}{2}$, 
 \item the property that their levels are minimal among the levels of all possible twists by characters $F^{\times}$. 
\end{itemize}

We observe that when \cite{BushnellH_06} 27.8 Proposition is applied to two cuspidal inducing data $\Xi_1, \Xi_2$ sharing the same underlying subgroup $J = \iota(E^{\times})I_F$, then instead of assumption (c) there, it suffices to assume that $\tr(g; \Xi_1) = \tr(g; \Xi_2)$ holds only for all $F$-minimal elements $g \in J$ with valuation of determinant equal to $-2\ell(\Xi_1)$. Indeed, the proof goes through verbatim. Now \eqref{eq:need_to_show_iso_cusp_types} follows from  \cite{BushnellH_06} 27.8 Proposition and Proposition \ref{prop:hard_traces}.
\end{proof}

In the above proof of the theorem we needed the following lemma.
\begin{lm} \label{lm:level_and_minimality_of_BH_side}
With notations as in the theorem, the central character of ${\rm BH}_{\theta, \psi, \alpha}$ is $\theta|_{F^{\times}}$, its level is $\frac{m+d}{2}$. For any character $\phi$ of $F^{\times}$, one has $0 < \ell({\rm BH}_{\theta, \psi, \alpha}) \leq \ell(\phi {\rm BH}_{\theta, \psi, \alpha})$.
\end{lm}
\begin{proof}
Clear from the construction of ${\rm BH}_{\theta, \psi, \alpha}$.
\end{proof}


\subsection{Traces of some minimal elements}\label{sec:hard_traces}

To complete the proof of Theorem \ref{thm:relation_ADLV_BH} we have to show the following proposition.

\begin{prop}\label{prop:hard_traces}
Assume that $\beta^{\vee}(\tilde{\theta}) = (\theta, \psi_{E,\alpha})$. For any $g \in \iota(E^{\times})I_F^{n+d}$ with $\ord_F(\det(g))$ odd, one has
\[
\tr(g; \Xi_{\tilde{\theta}}) = \tr \left(g; \cInd\nolimits_{\iota(E^{\times})I_F^{n+d}}^{\iota(E^{\times})I_F} \Lambda_{\theta, \psi, \alpha} \right).
\] 
\end{prop}

The rest of Section \ref{sec:hard_traces} is devoted to a proof of Proposition \ref{prop:hard_traces}. Central characters on both sides being equal, we may multiply $g$ by an appropriate central element and hence assume that $\ord_F(\det(g)) = 1$. Note that $I_F^{2(n+d)} = I_F^{m+d+1}$ acts trivial on both sides, hence we always may regard $g$ (and its constituents) modulo $I_F^{2(n+d)}$. Again, multiplying with an appropriate central element, we may assume that 
\[
g = g^{\prime} \iota(\unife) = u \iota(1 + \unife x) \iota(\unife), \quad \text{with $x \in \caO_F$ and $u \in I_F^{n+d}$}
\]
Recall the notation $\delta$ from Section \ref{sec:irred_and_cusp}. We may write 
\[
u = 1 + \matzz{u_1}{u_2}{u_3}{u_4} = 1 + \matzz{ \uniff^{ \frac{n+d+1-\delta}{2} } u_1^{\prime} }{ \uniff^{\frac{n+d - 1 + \delta}{2}} u_2^{\prime} }{ \uniff^{\frac{n+d +1 + \delta}{2} } u_3^{\prime} }{ \uniff^{\frac{n+d+1-\delta}{2}} u_4^{\prime}} 
\]
with some $u^{\prime}_i \in \caO_F$. We also have
\[
\matzz{g_1}{g_2}{g_3}{g_4} := g^{\prime} = u \iota(1 + \unife x) = \matzz{(1 + \Delta x) + u_1(1 + \Delta x) + u_2 \uniff  x}{x + u_1 x + u_2}{\uniff x + u_3 (1 + \Delta x)+ u_4\uniff x}{1 + u_3 x + u_4}.
\]
\noindent Let also
\[
j := g_2 \varepsilon\unife + g_1(a^{\prime} + 1 + \varepsilon).
\]
We use these notations until the end of Section \ref{sec:hard_traces}.


\subsubsection{Traces on the geometric side} 
\label{sec:trace_on_geom_side}

We use Proposition \ref{prop:traces_in_general} and notation from there. We introduce also the notation
\[
\ell := \left\lfloor \frac{n+d+1}{2} \right\rfloor = \frac{n + d + 1 - \delta}{2} \quad \text{ and } \quad \delta_{\ell} := \begin{cases} 0 & \text{if $\ell$ even,} \\ 1 & \text{if $\ell$ odd}. \end{cases}
\] 
For $a \in  (\fp/\fp^{n+d+m+1})^{\ast}$ we write $a = a^{\prime}\unife$ with $a^{\prime} \in U_E/U_E^{n+d+m}$.
\begin{lm}\label{lm:A_g}
The set $A_g$ from equation \eqref{eq:def_A_g} consists of exactly such $a = a^{\prime} \unife \in (\fp_E/\fp_E^{n+d+m+1})^{\ast}$, for which 
\[
(a^{\prime} + 1)(a^{\prime} + \varepsilon) \equiv 0 \mod \fp_E^{n+d}.
\]
holds. Thus,
\[ A_g = \begin{cases} \{ a \colon a^{\prime} \equiv 1 \mod \fp_E^n \} \dot{\cup} \{ a \colon a^{\prime} \equiv \varepsilon \mod \fp_E^n \} & \text{if $n > d$,} \\  \left\{ a \colon a^{\prime} \equiv 1 \mod \fp_E^{\ell} \right\} & \text{if $n \leq d$.} \end{cases} \]
\end{lm}

\begin{proof}
\end{proof}

Below we will compute explicit formulas for $t_a,r_a$. We will see that they and hence also the trace $\tr(g;\Xi_{\tilde{\theta}})$ only depend on $b \mod \fp_E^d$ (if $n < d$, it is even true that they only depend on $b \mod \fp_E^{\ell - 1 + \delta}$). Note that characteristic $2$ is used for that, in particular to deal with the monomials occurring in $t_a$ and containing $b^2$). Thus letting
\begin{equation}\label{eq:A_g_prime}
A_g^{\prime} = 
\begin{cases} \{\unife(1 + \unife^n b) \colon b \in \caO_E/\fp_E^d \} \cup \{\unife(\varepsilon + \unife^n b) \colon b\in \caO_E/\fp_E^d \} & \text{if $n > d$,} \\ 
\{\unife(1 + \unife^{\ell} b) \colon b \in \caO_E/\fp_E^{...} \} & \text{if $n \leq d$,} \end{cases}
\end{equation}
\noindent and regarding $b$ as an element in $\caO_E/\fp_E^d$, the multiplicity $q^m$ cancels with the term $\frac{1}{q^m}$ in the trace formula in Proposition \ref{prop:traces_in_general}, and we see that 
\begin{equation*}
\tr(g; \Xi_{\tilde{\theta}}) = \theta(\unife) \sum\limits_{a \in A_g^{\prime}} \tilde{\theta}(i(t_a,r_a))
\end{equation*}
\noindent where $t_a$ is given by the formulas \eqref{eq:ta_for_n_bigger_d}, \eqref{eq:second_eq_for_ta_n_geq_d} and \eqref{eq:second_eq_for_ta_n_smaller_d} and $r_a$ is given by \eqref{eq:ra_for_n_bigger_d} and \ref{lm:j_h_R_n_smaller_d}.

To compute $t_a,r_a$ explicitly, we will consider two cases: $n\geq d$ and $n<d$. Note that in contrast to what Lemma \ref{lm:A_g} let one guess, the case $n = d$ shows behavior similar to $n>d$. Lemma \ref{sublm:conj_to_iotaE} below might be seen as an explanation for this fact.


\noindent \underline{\textbf{Case $n \geq d$.}} Let $a = \unife a^{\prime} \in A_g$. By Lemma \ref{lm:A_g} we may assume that $a^{\prime} \equiv 1 \mod \fp_E^n$ or $a^{\prime} \equiv \varepsilon \mod \fp_E^n$. We only handle the first case, the second being completely analogous. By Lemma \ref{lm:A_g} we may write 

\begin{align*}
a^{\prime} &= 1 + \unife^n b  &\text{with $b \in \caO_E/\fp_E^{m + d + 1}$, and} \\
b &= \unife^{\delta} A + \unife^{1 - \delta} B, &\text{with $A \in \caO_F/\fp_F^{\frac{d+1}{2} - \delta + \frac{m}{2}}$ and $B \in \caO_F/\fp_F^{\frac{d-1}{2} + \delta + \frac{m}{2}}$.} 
\end{align*}

\begin{lm}\label{lm:j_h_R} We have
\[ \begin{aligned}
 j &= \varepsilon + \unife x \varepsilon^2 + (1 + \Delta x)\unife^n b + [u_1(\varepsilon + \unife x \varepsilon^2 + (1 + \Delta x)\unife^n b) + u_2 ( \varepsilon \unife + \uniff x(\varepsilon + \unife^n b))] \\
 &\equiv \varepsilon (1 + \unife x \varepsilon) \left( 1 + \frac{(1 + \Delta x)\unife^n b}{\varepsilon (1 + \unife x \varepsilon)} + u_1 + \unife u_2\right) \mod \fp_E^{m+1} \\
h(g,a) &\equiv  \frac{1 + \Delta x}{\varepsilon(1 + \unife \varepsilon x)} (\varepsilon_0 b + \unife^{n-d}b^2)\left( 1 + \frac{\unife^n b}{1 + \unife x} \right) + \dots \\
&\quad+ \unife^{1-\delta} \varepsilon^{\frac{n+d+1-\delta}{2}} (1 + \unife^n b)(u_1^{\prime} + u_4^{\prime}) + \unife^{\delta} \varepsilon^{\frac{n+d-1+\delta}{2}} (u_2^{\prime} + \varepsilon u_3^{\prime}) \mod \fp_E^{n+d}\\
R &\equiv \varepsilon_0 (1 + \uniff^{\frac{n+\delta}{2}} A) \mod \fp_E^{n+d}.
\end{aligned}
\]
\end{lm}
\begin{proof}
The first formula for $j$ is straightforward, the second follows using $m + 1 = 2n + d$. From it we deduce 
\[
j \equiv \varepsilon + \unife x \varepsilon^2 + \unife^n b \mod \fp_E^{n+d},
\]
and hence
\begin{equation}\label{eq:j-inverse_mod_n+d}
j^{-1} \equiv \varepsilon^{-1}(1 + \unife \varepsilon x)^{-1}\left( 1 + \frac{\unife^n b}{1 + \unife x} \right) \mod \fp_E^{n+d}.
\end{equation}
Further, 
\[
g_4 \varepsilon \unife + g_3(a^{\prime} + 1 + \varepsilon) = \varepsilon \unife + \uniff x \varepsilon + \uniff x \unife^n b + [u_3 (\varepsilon + \unife x \varepsilon^2 + (1 + \Delta x)\unife^n b) + u_4 (\varepsilon \unife + \uniff x(\varepsilon + \unife^n b))]
\]
By definition, $\unife^{n+d+1}h(g,a)j = g_4 \varepsilon \unife + g_3(a^{\prime} + 1 + \varepsilon) - j \unife (1 + \unife^n b)$. Using $\Delta = \unife + \varepsilon \unife$ we deduce 
\[
\begin{aligned}
\unife^{n+d+1}h(g,a)j &= \Delta \unife^n b (1 + \Delta x) + \unife (1 + \Delta x) \unife^{2n} b^2 + \dots \\
&+ u_1\unife(1 + \unife^n b)(\varepsilon + \unife x \varepsilon^2 + (1 + \Delta x)\unife^n b) + u_2 \unife^2 \varepsilon(1 + \unife^n b) (1 + \unife x(\varepsilon + \unife^n b)) + \dots \\
&+ u_3 (\varepsilon + \unife x \varepsilon^2 + (1 + \Delta x)\unife^n b) + u_4 \unife \varepsilon (1 + \unife x(\varepsilon + \unife^n b))
\end{aligned}
\]
Recall that $\varepsilon_0 = \unife^{-(d+1)}\Delta$. We deduce
\[
\begin{aligned}
h(g,a)j &= (1 + \Delta x) (\varepsilon_0 b + \unife^{n-d}b^2)  + \dots \\
&+ \unife^{1-\delta} \varepsilon^{\frac{n+d+1-\delta}{2}} u_1^{\prime} (1 + \unife^n b)(\varepsilon + \unife x \varepsilon^2 + (1 + \Delta x)\unife^n b) + \dots \\
&+ \unife^{\delta} \varepsilon^{\frac{n+d+1+\delta}{2}} u_2^{\prime} (1 + \unife^n b) (1 + \unife x(\varepsilon + \unife^n b)) + \dots \\
&+ \unife^{\delta} \varepsilon^{\frac{n+d+1+\delta}{2}} u_3^{\prime} (\varepsilon + \unife x \varepsilon^2 + (1 + \Delta x)\unife^n b) + \dots \\
&+ \unife^{1-\delta} \varepsilon^{\frac{n+d+1-\delta}{2} + 1} u_4^{\prime} (1 + \unife x(\varepsilon + \unife^n b))
\end{aligned}
\] 
Finally, we compute modulo $\fp_E^{n+d}$ (using \eqref{eq:j-inverse_mod_n+d}),
\[
\begin{aligned}
h(g,a) &\equiv \frac{1 + \Delta x}{\varepsilon(1 + \unife \varepsilon x)} (\varepsilon_0 b + \unife^{n-d}b^2)\left( 1 + \frac{\unife^n b}{1 + \unife x} \right)  + \dots \\
&+ \unife^{1-\delta} \varepsilon^{\frac{n+d+1-\delta}{2}} u_1^{\prime} (1 + \unife^n b) + \dots \\
&+ \unife^{\delta} \varepsilon^{\frac{n+d-1+\delta}{2}} u_2^{\prime} + \\
&+ \unife^{\delta} \varepsilon^{\frac{n+d+1+\delta}{2}} u_3^{\prime}  + \dots \\
&+ \unife^{1-\delta} \varepsilon^{\frac{n+d+1-\delta}{2}} u_4^{\prime} (1 + \unife^n b).
\end{aligned}
\] 
This is exactly the claimed formula for $h(g,a)$. The computation of $R$ is straightforward, by using $a^{\prime} = 1 + \uniff^{\frac{n+\delta}{2}} A + \unife \uniff^{\frac{n-\delta}{2}} B$.
\end{proof}

As by definition $\beta^{\prime}_{g,a} = \frac{\det(u\iota(1 + \unife x))}{j}$, we see that $t_a \in U_E/U_E^{m+1}$ from Proposition \ref{prop:traces_in_general} is determined by $a$ by the following formula,
\begin{equation}\label{eq:ta_for_n_bigger_d}
t_a = \frac{\varepsilon\det(u\cdot\iota(1 + \unife x))}{j(1 + \unife^n h(g,a)R^{-1})} = \frac{\det(u)(1 + \unife x)}{\frac{j}{\varepsilon (1 + \unife\varepsilon x)}(1+\unife^n R^{-1}h)}.
\end{equation}
A straightforward computation utilizing Lemma \ref{lm:j_h_R} shows that its denominator is
\begin{equation}\label{eq:second_eq_for_ta_n_geq_d}
\frac{j}{\varepsilon(1 + \unife \varepsilon x)}(1 + \unife^n R^{-1}h) = 1 + \unife^n \varepsilon_0^{-1} \frac{1 + \Delta x}{\varepsilon(1 + \unife \varepsilon x)}(\unife^{n+\delta}\varepsilon_0 A b + \unife^{n-d} b^2 + \unife^{2n-d+\delta}Ab^2) + U,
\end{equation}
where
\[
\begin{aligned}
U &= \unife^{n+d+1-\delta} \varepsilon^{\frac{n+d+1-\delta}{2}} u_1^{\prime} + \unife^{n+d+\delta} \varepsilon^{\frac{n+d-1+\delta}{2}} u_2^{\prime} + \unife^n \varepsilon_0^{-1}\left(1 + \unife^{n+\delta} A + \frac{\unife^n b}{1 + \unife x}\right)\cdot \dots \\
&\dots\cdot \left( \unife^{1- \delta} \varepsilon^{\frac{n+d+1-\delta}{2}}(1 + \unife^n b)(u_1^{\prime} + \unife_4^{\prime}) + \unife^{\delta} \varepsilon^{\frac{n+d-1+\delta}{2}}(u_2^{\prime} + \varepsilon\unife_3^{\prime}) \right)
\end{aligned}
\]
is the part depending on $u$. Further, Lemma \ref{lm:j_h_R} also implies that
\begin{equation}\label{eq:ra_for_n_bigger_d}
r_a = R^{-1}h(g,a) = \frac{1}{\varepsilon_0(1 + \unife x)} (\varepsilon_0 b + \unife^{n-d}b^2) + \varepsilon_0^{-1}\unife^{1-\delta} (u_1^{\prime} + u_4^{\prime}) + \varepsilon_0^{-1} \unife^{\delta} (u_2^{\prime} + u_3^{\prime})
\in \caO_E/\fp_E^d.
\end{equation}


\noindent \underline{\textbf{Case $n < d$.}} 
Let $a = \unife a^{\prime} \in A_g$. By Lemma \ref{lm:A_g} we may write 

\begin{align*}
a^{\prime} &= 1 + \unife^{\ell} b &\text{with $b \in \caO_E/\fp_E^{n + m + d + 1 - \ell}$, and} \\
b &= \unife^{\delta_{\ell}} A + \unife^{1 - \delta_{\ell}} B, &\text{with $A,B$ elements of appropriate subquotients of $\caO_F$.} 
\end{align*}

\begin{lm}\label{lm:j_h_R_n_smaller_d} We have
\begin{align*}
j &= \varepsilon + \unife x \varepsilon^2 + (1 + \Delta x)\unife^{\ell} b + [u_1(\varepsilon + \unife x \varepsilon^2 + (1 + \Delta x)\unife^{\ell} b) + u_2 ( \varepsilon \unife + \uniff x(\varepsilon + \unife^{\ell} b))] \\
 &\equiv \varepsilon (1 + \unife x \varepsilon) \left( 1 + \frac{(1 + \Delta x)\unife^{\ell} b}{\varepsilon (1 + \unife x \varepsilon)} + u_1 + \unife u_2 \right) \mod \fp_E^{m+1} \\
h(g,a) &\equiv  \frac{1 + \Delta x}{\varepsilon(1 + \unife \varepsilon x)} (\unife^{1 - \delta}b^2 + \unife^{\ell - n} \varepsilon_0 b) \left( 1 + \frac{\unife^{\ell} b}{1 + \unife x} \right) + \dots \\
&\quad+ \unife^{1-\delta} \varepsilon^{\frac{n+d+1-\delta}{2}} (1 + \unife^{\ell} b)(u_1^{\prime} + u_4^{\prime}) + \unife^{\delta} \varepsilon^{\frac{n+d-1+\delta}{2}} (u_2^{\prime} + \varepsilon u_3^{\prime}) \mod \fp_E^{n+d}\\
R &\equiv \varepsilon_0 (1 + \unife^{\ell + \delta_{\ell}} A) \mod \fp_E^{n+d}.
\end{align*}
\end{lm}

\begin{proof}
Straightforward computation, similar to the one in the proof of Lemma \ref{lm:j_h_R}.
\end{proof}

As in the case $n \geq d$, the element $t_a \in U_E/U_E^{m+1}$ from Proposition \ref{prop:traces_in_general} is given by the formula \eqref{eq:ta_for_n_bigger_d}, but now the denominator is 
\begin{equation}\label{eq:second_eq_for_ta_n_smaller_d}
\frac{j}{\varepsilon(1 + \unife \varepsilon x)}(1 + \unife^n R^{-1}h) = 1 + \unife^n \varepsilon_0^{-1} \frac{1 + \Delta x}{\varepsilon(1 + \unife \varepsilon x)}(\unife^{1-\delta}b^2 + \unife^{\ell}\cdot (\unife^{\delta_{\ell}}A) \cdot (\unife^{1 - \delta}b^2) + \unife^{2\ell - n}\varepsilon_0 (\unife^{\delta_{\ell}}A)b) + U,
\end{equation}
where
\[
\begin{aligned}
U &= u_1 + \unife u_2 + \unife^n \varepsilon_0^{-1}\left(1 + \unife^{\ell} (\unife^{\delta_{\ell}}A) + \frac{\unife^{\ell} b}{1 + \unife x}\right)\cdot \dots \\
&\dots\cdot \left( \unife^{1- \delta} \varepsilon^{\frac{n+d+1-\delta}{2}}(1 + \unife^{\ell} b)(u_1^{\prime} + \unife_4^{\prime}) + \unife^{\delta} \varepsilon^{\frac{n+d-1+\delta}{2}}(u_2^{\prime} + \varepsilon\unife_3^{\prime}) \right)
\end{aligned}
\]
is the part depending on $u$. Lemma \ref{lm:j_h_R_n_smaller_d} also show that
\[ r_a = R^{-1}h(1 + \unife^n R^{-1}h) \in \caO_E/\fp_E^d \]
with
\begin{align*}
 R^{-1}h &= \frac{\unife^{1- \delta}b^2}{\varepsilon_0 (1 + \unife x)}  + \frac{\unife^{\ell - n} b}{1 + \unife x} + \frac{\unife^{\ell + 1 - \delta}}{\varepsilon_0 (1+\unife x)^2}((1 + \unife x) \cdot (\unife^{\delta_{\ell}} A) \cdot b^2 + b^3) + \dots \\ 
 &\quad + \varepsilon_0^{-1}(1 + \unife^{\ell}(\unife^{1 - \delta_{\ell}} B))(u_1^{\prime} + u_4^{\prime}) + \varepsilon_0^{-1} \unife^{\delta}(1 + \unife^{\ell}(\unife^{\delta_{\ell}} A))(u_2^{\prime} + u_3^{\prime}) \\
 1 + \unife^n R^{-1}h &= 1 + \unife^n \varepsilon_0^{-1} \left(\frac{1}{1 + \unife x}(\unife^{1 - \delta} b^2 + \unife^{\ell - n} \varepsilon_0 b) + \unife^{1-\delta}(u_1^{\prime} + u_4^{\prime}) + \unife^{\delta}(u_2^{\prime} + u_3^{\prime})\right).
 \end{align*}


\subsubsection{Traces on the induced side}
Mackey formula gives:
\[
\tr\left(g; \Ind\nolimits_{\iota(E^{\times})I_F^{n+d}}^{\iota(E^{\times})I_F} \Lambda \right) = \sum_{y, \lambda} \Lambda\left(r_{y,\lambda} g r_{y,\lambda}^{-1}\right), 
\]
where the sum is taken only over such $y,\lambda$, for which $r_{y,\lambda} g r_{y,\lambda}^{-1} \in \iota(E^{\times})I_F^{n+d}$. The following lemma is true also for the wildly ramified $E/F$:
\begin{lm}[\cite{Ivanov_15_ram}, Lemma 5.15]
The elements
\[
r_{y,\lambda} := \matzz{1}{0}{0}{y} \matzz{1}{\lambda}{0}{1} \quad \text{with $y \in U_F/U_F^{\frac{n+d+1-\delta}{2}}$, $\lambda \in \caO_F/\fp_F^{\frac{n+d-1+\delta}{2}}$}
\] 
(where $y,\lambda$ are chosen to be fixed preimages in $U_F$ resp. $\caO_F$) form a set of coset representatives in $I_F$ for $\iota(E)I_F/\iota(E) I_F^{n+d} = I_F/\iota(U_E)I_F^{n+d}$.
\end{lm}
We compute:
\[
\begin{aligned}
r_{y,\lambda} g r_{y,\lambda}^{-1} \iota(\unife)^{-1} = (r_{y,\lambda} u r_{y,\lambda}^{-1}) r_{y,\lambda} \iota(1 + \unife x)\iota(\unife) r_{y,\lambda}^{-1} \iota(\unife)^{-1},
\end{aligned}
\]
with 
\begin{equation}\label{eq:commutator_r_iota_r}
\begin{aligned}
r_{y,\lambda} &\iota(1 + \unife x)\iota(\unife) r_{y,\lambda}^{-1} \iota(\unife)^{-1} =\\ &\matzz{y^{-1}(1 + \Delta x)(1+ \lambda \Delta + \lambda^2\uniff)}{x + (1 + \Delta x)(\uniff^{-1}\Delta + \lambda + y^{-1} \uniff^{-1}\Delta (1+ \lambda \Delta + \lambda^2\uniff)) }{\uniff(x + \lambda(1 + \Delta x))}{\Delta x + (1 + \Delta x)(y + \Delta \lambda)}
\end{aligned}
\end{equation}
and
\begin{equation}\label{eq:Du}
D_{u,y, \lambda} := r_{y,\lambda} u r_{y,\lambda}^{-1} = \matzz{1 + u_1 + \lambda u_3}{y^{-1}(u_2 + \lambda(u_1 + u_4) + \lambda^2 u_3)}{yu_3}{1 + u_4 + \lambda u_3}.
\end{equation}

We investigate the contribution of $r_{y,\lambda} g r_{y,\lambda}^{-1}$ to the trace on the induced side. Unless $r_{y,\lambda} g r_{y,\lambda}^{-1} \in \iota(E^{\times})I_F^{n+d}$, this contribution is zero, so we may assume this. 
\begin{lm}\label{sublm:conj_to_iotaE} One has 
\[ r_{y,\lambda} g r_{y,\lambda}^{-1} \in \iota(E^{\times})I_F^{n+d} \LRar 
\begin{cases} y \equiv 1 \mod \fp_F^{\frac{n+\delta}{2}} \text{ and }\lambda \equiv 0 \text{ or } \uniff^{-1}\Delta \mod \fp_F^{\frac{n-\delta}{2}} & \text{if $n \geq d$,} \\ 
y \equiv 1 \mod \fp_F^{\frac{\ell + \delta_{\ell}}{2}} \text{ and } \lambda \equiv 0 \mod \fp_F^{\frac{\ell - \delta_{\ell}}{2}} & \text{if $n < d$.}
\end{cases}
\]
\end{lm}
\begin{proof} 
By normality of $I_F^{n+d}$ in $I_F$, $r_{y,\lambda} g r_{y,\lambda}^{-1} \in \iota(E^{\times})I_F^{n+d}$ is equivalent to $r_{y,\lambda} \iota(1 + \unife x)\iota(\unife) r_{y,\lambda}^{-1} \in \iota(U_E)I_F^{n+d}$. This last is equivalent to the existence of $c_0,c_1 \in \caO_F$ with 
\[
r_{y,\lambda} \iota(1 + \unife x)\iota(\unife) r_{y,\lambda}^{-1} \iota(\unife)^{-1} \equiv \matzz{c_0 + \Delta c_1}{c_1}{\uniff c_1}{c_0} \mod I_F^{n+d} = \matzz{1 + \fp_F^{\frac{n+d+1-\delta}{2}}}{\fp_F^{\frac{n+d-1+\delta}{2}}}{\fp_F^{\frac{n+d+1+\delta}{2}}}{1+\fp_F^{\frac{n+d+1-\delta}{2}}}
\]
Utilizing \eqref{eq:commutator_r_iota_r} and comparing the lower rows, we in particular must have 
\[
\begin{aligned}
c_0 &\equiv \Delta x + (1 + \Delta x)(y + \Delta \lambda) \mod \fp_F^{\frac{n+d+1-\delta}{2}} \\
c_1 &\equiv x + \lambda(1 + \Delta x) \mod \fp_F^{\frac{n+d-1+\delta}{2}}
\end{aligned}
\]
Comparing with the upper rows and simplifying we deduce the equations
\[
\begin{aligned}
y &\equiv 1 + \lambda \Delta + \uniff\lambda^2 \mod \fp_F^{\frac{n+\delta}{2}} \\
y^2 &\equiv 1 + \lambda \Delta + \uniff\lambda^2  \mod \fp_F^{\frac{n+d+1-\delta}{2}}.
\end{aligned}
\]

Now assume $n \geq d$. Taking the sum of the two equations above we deduce that $y \equiv 1 \mod \fp_F^{\frac{n+\delta}{2}}$. Putting this into the second equation, we deduce $\Delta \lambda + \uniff\lambda^2 \equiv 0 \mod \fp_F^{\frac{n+d+1-\delta}{2}}$. Note that this implies $\lambda \equiv 0 \mod \fp_F^{\frac{d-1}{2}}$ (assuming the contrary easily leads to a contradiction). Thus we may write $\lambda = \uniff^{\frac{d-1}{2}}\lambda_0$ and the second equation is seen to be equivalent to $\lambda_0^2 + \varepsilon^{-\frac{d+1}{2}}\varepsilon_0 \lambda_0 \equiv 0 \mod \fp_F^{\frac{n-d+1-\delta}{2}}$, from which the claim follows (note that $\uniff^{-1}\Delta = \uniff^{\frac{d-1}{2}}\varepsilon^{-\frac{d-1}{2}}\varepsilon_0$). 

The case $n < d$ is done similarly. \qedhere
\end{proof}

We can find a (non-canonical) decomposition
\[
r_{y,\lambda} g r_{y,\lambda}^{-1} = D \cdot \iota(C_{x,y, \lambda}) \iota(\unife),
\]
\noindent with $D \in I_F^{n+d}$, $C_{x,y, \lambda} \in U_E$. As seen from the explicit computation in Lemma \ref{sublm:conj_to_iotaE}, we may take 
\begin{equation}\label{eq:def_of_C_n_geq_d}
C_{x,y,\lambda} := c_0 + \unife c_1
\end{equation}
with 
\[
\begin{aligned}
c_0 &:= \Delta x + (1 + \Delta x)(y + \Delta \lambda) = 1 + (1 + \Delta x)(1 + y + \Delta \lambda) \\
c_1 &:= x + \lambda(1 + \Delta x).
\end{aligned}
\]

In particular, we have $c_0 + \Delta c_1 = y(1 + \Delta x)$. We compute 
\begin{align*}
 \N_{E/F}(C_{x,y, \lambda}) &= c_0 (c_0 + \Delta c_1) + \uniff c_1^2 \\
 &\equiv 1 + \Delta x + \uniff x^2 + (1 + \Delta x)(\uniff (1 + \Delta x) \lambda^2 \\ 
 &+ (1 + \Delta x)(1 + y)^2 + \Delta x (1 + y) + \Delta\lambda (1 + \Delta x) + \Delta \lambda(1 + y)) \mod \fp_E^{2(n+d)}.
\end{align*}

We can decompose further, $D = D_{u,y, \lambda} D_{x,y, \lambda}$ with $D_{u,y, \lambda}$ as in \eqref{eq:Du} and 
\[
\begin{aligned}
D_{x,y, \lambda} &= r_{y,\lambda} \iota(1 + \unife x)\iota(\unife) r_{y,\lambda}^{-1} \iota(\unife)^{-1} \iota(C)^{-1} \\
&= \matzz{D_1\N_{E/F}(C)^{-1}}{D_2\N_{E/F}(C)^{-1}}{0}{1} 
\end{aligned}
\]
with 

\[
\begin{aligned}
\frac{D_1}{N_{E/F}(C_{x,y, \lambda})} - 1 &\equiv \frac{1+\Delta x}{(1 + \unife x)(1 + \unife \varepsilon x)}\left( \uniff(1+\Delta x) \lambda^2 + (1+\Delta x)(1+y)^2 + \Delta x (1+y) \right. \\
&+ \left. \Delta(1+\Delta x)\lambda + \Delta \lambda(1+y) \right) \mod \fp_E^{2(n+d)}. \\
\frac{D_2}{N_{E/F}(C_{x,y, \lambda})} &\equiv \frac{y^{-1}(1 + \Delta x)}{(1 + \unife x)(1 + \unife \varepsilon x)}\left( (1+y)^2 x + (1+y)^2 \lambda + (1+y)^2 \frac{\Delta}{\uniff} + \uniff x \lambda^2 + \uniff \lambda^3 + \Delta x \lambda \right. \\ 
&+ \left. \frac{\Delta}{\uniff} (1 + \Delta x) (1+ y+ \Delta \lambda) \right) \mod \fp_E^{2(n+d)}.
\end{aligned}
\]

The following lemma is immediate.

\begin{lm}\label{lm:psi_E_psi_relation}
Let $z = z_0 + \unife \varepsilon z_1 \in \fp_E^n/\fp_E^{m+1}$ with $z_0, z_1 \in F$. We have 
\[
\psi_{E,\alpha}(z) = \psi(\Delta(\alpha_1 z_0 + \alpha_0 z_1)). 
\]
\end{lm}

Using Lemma \ref{lm:psi_E_psi_relation}, we compute the contribution of $D_{x,y,\lambda}$,
\[
\begin{aligned}
\psi_{\fM,\iota(\alpha)}(D_{x,y, \lambda}) &= \psi_{\fM}(\iota(\alpha)(D_{x,y, \lambda} - 1)) \\
&= \psi_{\fM}\left(\matzz{\alpha_0 + \Delta \alpha_1}{\alpha_1}{\uniff\alpha_1}{\alpha_0} \matzz{N_{E/F}(C)^{-1}D_1 - 1}{N_{E/F}(C)^{-1}D_2}{0}{0}\right) \\
&= \psi\left(\alpha_0\left(\frac{D_1}{N_{E/F}(C_{x,y, \lambda})} - 1\right) + \alpha_1\left(\Delta\left(\frac{D_1}{N_{E/F}(C_{x,y, \lambda})} - 1\right) + \uniff \frac{D_2}{N_{E/F}(C_{x,y, \lambda})}\right)\right) \\
&= \psi_{E,\alpha}\left(\left(\frac{D_1}{N_{E/F}(C_{x,y, \lambda})} - 1\right) + \frac{\uniff}{\Delta} \frac{D_2}{N_{E/F}(C_{x,y, \lambda})} + \frac{\unife \varepsilon}{\Delta}\left(\frac{D_1}{N_{E/F}(C_{x,y, \lambda})} - 1\right)\right) \\
&= \psi_{E,\alpha}\left(\underbrace{\frac{\unife}{\Delta}\left(\frac{D_1}{N_{E/F}(C_{x,y, \lambda})} - 1\right) + \frac{\uniff}{\Delta} \frac{D_2}{N_{E/F}(C_{x,y, \lambda})}}_{=: z_{x,y, \lambda}} \right),
\end{aligned}
\]
and analogously the contribution of $D_{u,y,\lambda}$,
\[
\psi_{\fM,\iota(\alpha)}(D_{u,y, \lambda}) = \psi_{E,\alpha}\left( \underbrace{u_1 + \lambda u_3 + \Delta^{-1}yu_3  + \frac{\uniff}{\Delta}y^{-1}(u_2 + \lambda(u_1 + u_4) + \lambda^2 u_3) + \frac{\unife \varepsilon}{\Delta}(u_1 + u_4)}_{=: z_{u,y, \lambda}} \right).
\]

Thus we compute
\begin{align*}
\Lambda_{\theta,\psi,\alpha}(r_{y,\lambda} g r_{y,\lambda}^{-1}) 
&= \Lambda(D_{u,y, \lambda} D_{x,y, \lambda} \iota(C_{x,y, \lambda})\iota(\unife)) \\ 
&= \psi_{\fM,\iota(\alpha)}(D_{u,y, \lambda}) \psi_{\fM,\iota(\alpha)}(D_{x,y, \lambda}) \theta(C_{x,y, \lambda}) \theta(\unife) \\
&= \psi_{E,\alpha}(z_{x,y, \lambda} + z_{u,y, \lambda})\theta(C_{x,y, \lambda}) \theta(\unife) \\
&= (\theta,\psi_{E,\alpha})(C_{x,y, \lambda},z_{x,y, \lambda} + z_{u,y, \lambda})\theta(\unife) \\
&= \tilde{\theta}(\beta(C_{x,y, \lambda}, z_{x,y, \lambda} + z_{u,y, \lambda}))\theta(\unife),
\end{align*}

\noindent Here, $\beta$ is the isomorphism from Proposition \ref{prop:beta}. The fourth equation is by definition of the character $(\theta,\psi_{E,\alpha})$ of $\Pi$, and the fifth equation is by the assumption in Proposition \ref{prop:hard_traces}. Thus to show Proposition \ref{prop:hard_traces} it suffices to show that there is a bijection of sets,
\begin{equation}\label{eq:gamma_bijection}
\gamma \colon \left\{ y,\lambda \colon r_{y,\lambda} g r_{y,\lambda}^{-1} \in \iota(E^{\times}) I_F^{n+d}  \right\} \stackrel{\sim}{\longrar} A_g^{\prime},
\end{equation}
with $A_g^{\prime}$ as in \eqref{eq:A_g_prime}, such that
\begin{equation}\label{eq:beta_main_equality}
\beta(C_{x,y, \lambda}, z_{x,y, \lambda} + z_{u,y, \lambda}) = i(t_{\gamma(y,\lambda)}, r_{\gamma(y,\lambda)}).
\end{equation}
where $t_a,r_a$ are as in Section \ref{sec:trace_on_geom_side}. Using Lemma \ref{sublm:conj_to_iotaE} we may write
\begin{align*}
y &= \begin{cases}
      1 + \unife^n y_1 & \text{if $n \geq d$ and $y \equiv 1 \mod \fp_E^n$,} \\
      1 + \unife^{\ell} y_1 & \text{if $n < d$,}
     \end{cases} \\
\lambda &= \begin{cases}
           \unife^{n - 1} \varepsilon^{-1} \lambda_1 & \text{if $n \geq d$,} \\
           \unife^{\ell - 1} \varepsilon^{-1} \lambda_1 & \text{if $n < d$,}.
           \end{cases}
\end{align*}
for appropriate $y_1$ and $\lambda_1$. Now the appropriate bijection $\gamma$ is given as follows.
\begin{itemize}
\item Let $n \geq d$. We set $\gamma(y,\lambda) := \unife(1 + \unife^n (y_1 + \lambda_1))$. Thus $b$ corresponds to $y_1 + \lambda_1$, $\unife^{\delta}A$ corresponds to $y_1$, and $\unife^{1-\delta}B$ corresponds to $\lambda_1$. Analogously $\gamma$ can be defined for $y \equiv \varepsilon \mod \fp_E^n$.
\item Let $n < d$. We set $\gamma(y,\lambda) := 1 + \unife^{\ell}(y_1 + \lambda_1)$. Thus $b$ corresponds to $y_1 + \lambda_1$, $\unife^{\delta_{\ell}}A$ corresponds to $y_1$, and $\unife^{1-\delta_{\ell}}B$ corresponds to $\lambda_1$. 
\end{itemize}

\noindent Now a completely straightforward (but quite lengthy) computation shows that \eqref{eq:beta_main_equality} indeed holds with respect to this $\gamma$. This finishes the proof of Proposition \ref{prop:hard_traces}.

\subsection{Discussion of the relation to the Galois side}\label{sec:rectifier}

It is more than natural to ask, what the image of $R_{\tilde{\theta}}$ under the Langlands-correspondence is. Unfortunately, it is not clear how to characterize it in terms of $\tilde{\theta}$ and the geometry of $X_{\dot{w}}^m(1)$. 

At least in the tame case, for any character $\chi$ of $E^{\times}$, there is a character $\Delta_{\chi}$, the \emph{rectifier} of $\chi$, which controls the difference between the local Langlands correspondence and the two natural parametrizations 
\[
BH_{\chi} \mapsfrom \chi \mapsto \Ind\nolimits_{E/F} \chi 
\]
of the Galois and of the automorphic side by characters  various degree $2$-extensions $E/F$, where $BH_{\chi}$ denotes the representation attached to $\chi$ via theory of types, and $\Ind_{E/F} \chi$ denotes the induction to the Weil group of $F$ from the Weil group of $E$ of the character associated to $\chi$ by the local reciprocity isomorphism. The rectifier $\Delta_{\chi}$ is then uniquely determined by 
\[
{\rm LLC} \colon {\rm BH}_{\Delta_{\chi} \chi} \leftrightarrow \Ind\nolimits_{E/F} \chi
\]
under the local Langlands correspondence. For $\GL_2$ and unramified $E/F$, the rectifier is trivial on the units $U_E$, and equal to $(-1)$ on an uniformizer (in particular, it does not depend on $\chi$). This $(-1)$ shows up in the cohomology of the Deligne-Lusztig constructions attached to unramified tori as a $q$-power multiple of the Frobenius eigenvalue in the cohomology \cite{BoyarchenkoW_16,Chan_16_semiinf,Ivanov_15_unram}. But in the totally tamely ramified case, it is quite involved (cf. \cite{BushnellH_06} \S34.4). In \cite{Weinstein_09}, Weinstein had the very nice idea to recover the rectifier in the geometry by some twists of the action on the geometric structures (see \cite{Weinstein_09} Section 5). It is not clear how one can achieve a similar twist of the geometric actions for the extended affine Deligne-Lusztig varieties considered in \cite{Ivanov_15_ram} and here. 


\subsection{Remark on the tame case} \label{sec:remark_on_tame_case}

First we remark that the parametrization \eqref{eq:explicit_Cvm_param} used in the proof of Proposition \ref{prop:eADLV_higher_level} is better adapted (than parametrization (3.2) in \cite{Ivanov_15_ram}) also in the case of a totally tamely ramified torus (i.e., if $\charac F > 2$) -- the formulae get considerably easier. Moreover, the proof of \cite{Ivanov_15_ram} Theorem 4.2 can be simplified -- \cite{Ivanov_15_ram} Theorem 4.18 is in fact not necessary. Indeed, to prove it, it is (exactly as in the proof of Theorem \ref{thm:relation_ADLV_BH} in the present article) sufficient to just compare the traces of the elements lying in $\iota(E^{\times})$, which have $E$-valuation $1$. Nevertheless, \cite{Ivanov_15_ram} Theorem 4.18 is also interesting in its own right, as it describes completely the restriction of the cuspidal inducing datum $\Xi_{\chi}$ to the torus $\iota(E^{\times})$.

\bibliography{bib_ADLV}{}
\bibliographystyle{alpha}

\end{document}